\documentclass[11pt,reqno]{amsproc}
\usepackage[margin=1in]{geometry}
\usepackage{amsmath, amsthm, amssymb}
\usepackage{times, esint,stackrel}
\usepackage{enumitem}
\usepackage{color}  
\usepackage[colorlinks=true]{hyperref}
\usepackage[color=yellow]{todonotes}
\usepackage{mathtools}
\usepackage{etoolbox}
\usepackage{mathrsfs}
\usepackage{framed}

\makeatletter
\patchcmd\@thm
  {\let\thm@indent\indent}{\let\thm@indent\noindent}%
  {}{}
\makeatother

\usepackage{etoolbox}
\expandafter\patchcmd\csname\string\proof\endcsname
  {\normalparindent}{0pt }{}{}

\newcommand{\subscript}[2]{$#1 _ #2$}

\newtheorem*{thm*}{Theorem}

\newcommand{\be}{\begin{equation}}
\newcommand{\ee}{\end{equation}}
\newcommand{\bea}{\begin{eqnarray}}
\newcommand{\eea}{\end{eqnarray}}

\newtheorem{thm}{Theorem}
\newtheorem{prop}{Proposition}
\newtheorem{lemma}{Lemma}
\newtheorem{cor}{Corollary}

\theoremstyle{definition}
\newtheorem{rem}{Remark}

\newcommand{\rmd}{{\rm d}}

\newcommand{\ol}[1]{\mkern 1.5mu\overline{\mkern-1.5mu#1\mkern-1.5mu}\mkern 1.5mu}

\def\wc{\rightharpoonup}

\def\Re{\textup{\textbf{Re}}}
\def\Wi{\textup{\textbf{Wi}}}
\def\St{\boldsymbol{\tau}}
\def\alp{\boldsymbol{\alpha}}

\newcommand{\bq}{\begin{equation}}
\newcommand{\eq}{\end{equation}}
\newcommand{\bqa}{\begin{eqnarray*}}
\newcommand{\eqa}{\end{eqnarray*}}

\newcommand{\p}{\partial}

\def\XXint#1#2#3{{\setbox0=\hbox{$#1{#2#3}{\int}$ }
\vcenter{\hbox{$#2#3$ }}\kern-.6\wd0}}

\title[Polymer Drag Reduction]{Boundary conditions and polymeric drag\\reduction for the Navier-Stokes
equations}
\author{Theodore D. Drivas and Joonhyun La}
\address{Department of Mathematics, Princeton University, Princeton, NJ 08544\vspace{-2mm} }

\address{\textbf{Current Address}: Department of Mathematics, Stony Brook University, Stony Brook, NY 11794}
\email{tdrivas@math.princeton.edu}
\email{tdrivas@math.stonybrook.edu}
\address{Department of Mathematics, Princeton University, Princeton, NJ 08544\vspace{-2mm}}
\address{\textbf{Current Address}: Department of Mathematics, Stanford University, Stanford CA 94305}
\email{joonhyun@math.princeton.edu}
\email{joonhyun@stanford.edu}
\date{today}

\date{today}
\begin{document}

\begin{abstract} 
Reducing wall drag in turbulent pipe and channel flows is an issue of great practical importance.
 In engineering applications, end-functionalized polymer chains are often employed as agents to reduce drag.
These are polymers which are floating in the solvent and attach (either by adsorption or through irreversible chemical binding) at one of their chain ends to the substrate (wall). We propose a PDE model to study this setup in the simple setting where the solvent is a viscous
incompressible Navier-Stokes fluid occupying the bulk of a smooth domain $\Omega\subset \mathbb{R}^d$, and the 
wall-grafted polymer is in the so-called mushroom regime (inter-polymer spacing on the order of the typical polymer length).  The microscopic description of the polymer enters into the macroscopic description of the fluid motion
through a dynamical boundary condition on the wall-tangential stress of the fluid, something akin to (but distinct from) a history-dependent slip-length. 
We establish global well-posedness of strong solutions in two-spatial dimensions and prove that the inviscid limit to the strong Euler solution holds with a rate.  Moreover, the wall-friction factor $\langle f\rangle$ and the global energy dissipation $\langle \varepsilon\rangle$ vanish inversely proportional to the Reynolds number $\Re$.  This scaling corresponds to Poiseuille's  law  for the friction factor $\langle f\rangle \sim1/\Re$ for laminar flow and thereby quantifies
drag reduction in our setting.  These results are in stark contrast to those available for physical boundaries without polymer additives modeled by, e.g.,  no-slip conditions, where no such results are generally known even in two-dimensions. 
\end{abstract}

\maketitle

\section{Introduction }

 The problem of reducing dissipation and drag in flows confined by solid boundaries
  is a classical one which is of great importance for practical engineering applications.
Remarkably,  in 1948 Toms \cite{T48} discovered that the addition of small amounts of polymer (for example, 5-10 ppm per weight) to a turbulent flow is known to have
a pronounced effect on reducing friction drag \cite{PDRreview,Swhite00}. This phenomenon -- now called \emph{polymer drag reduction} -- 
is widely employed in practice,  has had a long record of success and remains a subject of active research \cite{PRLdrag}.
  However, our theoretical
understanding lags behind and there is not a consensus on which properties of the polymer are most
critical for this behavior.

Drag reduction is most evident in turbulent boundary layers, in which dissolving trace
quantities of long-chain flexible polymers into solution can reduce turbulent friction
losses by up to  80\% relative to solvent alone \cite{Virk}.  Moreover, even when
dissolved in the solvent bulk the boundary effects may be nontrivial since it is known that polymers can
spontaneously adsorb from solution onto surfaces if the interaction between the
polymer and the surface is more favorable than that of the solvent \cite{Abs99,NA03}.
These facts suggest that the essential mechanism for drag reduction occurs near solid boundaries.

To take advantage of this,  so-called \emph{end-functionalized} polymers
are often employed in industrial and technological applications.  These are polymers which are attached at one end to the bounding wall,
with the rest of the polymer being relatively neutral to the substrate (neither attracted nor repelled).  
End-functionalized polymers can occur either from polymer adsorption or be created by an irreversible 
attachment facilitated by chemically binding one end of the polymer to the wall. The latter, known as 
{grafted polymer chains}, has the practical advantage that the polymer does not wash out as a consequence of the flow of the solvent.  On the other hand, adsorption is easier to create from a technical point
of view.  For a detailed discussion of these two situations, see Chapters 4 and 13 of \cite{NA03} and \S VIII of \cite{NA01}.

To model end-functionalized polymers mathematically, we introduce a new boundary condition for the Navier-Stokes equations.   
Our model describes the situation in which the polymer ends are fixed on the wall and do not move with the solvent, applying
either to polymers which are irreversibly grafted or in situations where the adsorption is sufficiently strong. 
We also assume the polymer along the wall is in the so-called ``mushroom regime", i.e. they are spaced sufficiently far for them to interact only weakly.
 The polymer is felt by the fluid through a tangential stress balance, which is at equilibrium.  Specifically,  the stress that the wall bound polymer-fluid layer exerts on the bulk fluid is equal to the the viscous stress the bulk fluid exerts back onto the layer.
The main effect of our new boundary condition is that the presence of polymer allows the fluid to slip along the solid walls with effectively constant (in viscosity) slip-length.
We remark that modeling the effect of polymer by an effective slip length is a well-established idea \cite{BdG92,DHL97} which remains practically very effective \cite{Karniadakis1,Karniadakis2}.  
Our contribution to these ideas is to provide a rational derivation of such an effect from a kinetic theoretic description under a
number of simplifying assumptions.  As a result, we show that -- at the PDE level in the regime we study -- the boundary condition is not the usual Navier
condition but rather it is dynamical and appears as an evolution for the tangential fluid stress along the walls. 
Assuming a bead-spring approximation with Hookean dumbbell potential describes the polymer, our macroscopic system is as follows
\begin{align}\label{NSbintro}
\partial_t u^\nu + u^\nu \cdot \nabla u^\nu  &= - \nabla p^\nu  + \nu \Delta u^\nu  + f_b\quad\quad \ \  \text{in} \quad   \Omega \times (0,T), \\
u^\nu|_{t=0} &= u_0  \qquad\qquad\qquad \quad\quad\quad \ \ \   \text{on} \quad \!\!  \Omega \times \{t=0\},\\ 
\nabla \cdot u^\nu  &= 0  \qquad\qquad\qquad \quad\quad\quad\quad \ \text{in} \quad   \Omega \times [0,T),\\ \label{nonpen}
u^\nu \cdot \hat{n} &= 0  \qquad\qquad\qquad \quad\quad\quad\quad \ \text{on} \quad \!  \partial \Omega \times [0,T),\\ 
\left(\partial_t  +  \frac{4H k_B {\overline{T}}}{R \zeta } \right)\left(2(D(u^\nu)  \hat{n})\cdot \hat\tau_i  + \frac{1}{2R}u^\nu\cdot \hat{\tau}_i\right)&= - \frac{ {k_B \overline{T}} N_P }{\rho \nu R} u^\nu \cdot \hat\tau_i  \quad\quad\ \ \  \ \ \ \  \ \text{on} \quad \!  \partial \Omega \times (0,T),\label{NSfintro}\\
 &\hspace{48mm}  i=1,  \dots, d-1 \nonumber
\end{align}
 where $D(u)= 1/2(\nabla u +( \nabla u)^t)$, $f_b$ is a body force, and for every $x\in \p\Omega$, the vectors $\{ \hat{\tau}_i(x)\}_{i=1}^{d-1}$ form an orthogonal basis of the tangent space of $\partial\Omega$ at $x$.   See \S \ref{nondim} for a non-dimensionalization of these equations. The physical constants appearing in the system \eqref{NSbintro}--\eqref{NSfintro} are
  \begin{itemize}
 \item  $\nu$, the kinematic viscosity of the fluid (solvent),
  \item  $k_B$,  the Boltzmann constant,
   \item $\overline{T}$, the absolute temperature,
 \item  $R$, the characteristic length-scale of the polymer,
 \item  $N_P $, the number density of the grafted polymer carpet on the wall (see Eq. \eqref{rhoL}),
 \item $H$, the (non-dimensional) spring constant of the Hookean polymer,
 \item $\zeta$, the bead friction coefficient,
 \item   $\rho$, the mass density of the fluid (solvent).
 \end{itemize}
See \S 2 along with standard texts \cite{Doi,BA87a,BA87b,Ottinger12} for specifics on these parameters.  We here note only that the Stokes-Einstein relation in three dimensions\footnote{In two-dimensions, a relation of the type \eqref{SE} is not well established, although there has been some recent work in the setting of hard disks \cite{SErel1}. For general spatial dimensions $d\geq 3$, the Stokes-Einstein relation reads \cite{SErel}
\be\nonumber
\zeta= \frac{4d\pi^{d/2}}{(d-1)\Gamma\left(\frac{d}{2}-1\right)} \rho \nu a^{d-2}.
\ee}
describes the relation between $\zeta$ and  $\nu$ via the formula
\begin{align} \label{SE}
\zeta = 6 \pi \rho \nu a,  
\end{align}
where we have introduced
  \begin{itemize}
 \item   $a$, the bead width in the coarse-grained bead-spring polymer description.
 \end{itemize}
We remark that the relation \eqref{SE} is crucial in the study of the vanishing viscosity limit, as we will see in \S \ref{sec:dragred}.

The system  \eqref{NSbintro}--\eqref{NSfintro} and, in particular, the boundary condition \eqref{NSfintro}  is a special case of a more general system which accommodates  end-functionalized polymers described by non-Hookean spring potentials. However we do not pursue the mathematical questions of existence of solutions and inviscid limit of those other models here. The main theorem of our paper establishes global well-posedness  of the Navier-Stokes -- End-Functionalized system \eqref{NSbintro}--\eqref{NSfintro} and asymptotics of the resulting flow at large Reynolds number $\Re$. The result, which summarizes and combines Theorems \ref{thmglobal} and \ref{IL} in the main body, states
\begin{thm*}
Let $\Omega\subset \mathbb{R}^2$ be a planar bounded domain with smooth boundary. For any $T>0$, there exists a unique global  strong solution  $u^\nu$ of \eqref{NSbintro}--\eqref{NSfintro} with smooth initial data $u_0$ on $[0,T]\times \Omega$.  Let  $u$ be the global strong Euler solution with the same initial data. Then $u^\nu\to u$ strongly in $C(0,T;L^2(\Omega))$ as $\Re= VL/\nu\to \infty$. 
Furthermore, the wall-friction factor $\langle f\rangle$ tends to zero inversely with the Reynolds number 
\be
 \langle f\rangle := \nu \int_0^T \frac{1}{|\Omega|} \int_{\partial\Omega} \hat{n} \cdot \nabla u^\nu(x,t) \rmd S\rmd t, \qquad \frac{\langle f\rangle}{V}=   O(\Re^{-1}),
\ee
and likewise the energy dissipation $\langle \varepsilon^\nu \rangle$ tends vanishes as
\be
\langle \varepsilon^\nu \rangle := \fint_0^T \fint_\Omega \nu |\nabla u^\nu(x,t)|^2 \rmd x \rmd t, \qquad\qquad\ \ \frac{\langle \varepsilon^\nu \rangle}{V^3/L}  = O(\Re^{-1}).
\ee
\end{thm*}
\noindent In the statement of the Theorem, the Reynolds number can be taken large either by reducing viscosity $\nu$ or increasing characteristic velocity $V$, with all other parameters fixed.  See Remark \ref{restinvlim} for a discussion of these different limits physically, as well as their limitations.

This result should be contrasted with the situation without polymer.  The two most commonly used boundary conditions for neutral Navier-Stokes fluids are the so-called no-slip and Navier-friction (with viscosity dependent slip-length) conditions.  No-slip, or stick, boundary conditions correspond to the situation in which that the fluid velocity matches the boundary velocity (which we here consider stationary):
\begin{align} \label{noslip}
u^\nu = 0 \quad  \text{on}  \quad  \partial \Omega \times (0,T).
\end{align}
On the other hand, the Navier-friction boundary conditions\footnote{We remark that, in steady-state, the boundary condition  \eqref{NSfintro}  which we propose to describe wall-grafted polymers reduces to a Navier-friction condition \eqref{NSslip1} with a slip-length defined by characteristics of the polymer additives and fluid solvent.} combine non-penetration \eqref{nonpen} together with
\begin{align} \label{NSslip1}
2\nu (D(u^\nu)  \hat{n})_{\tau_i}&= -\alpha u^\nu\cdot \hat\tau_i \quad \text{on}  \quad  \partial \Omega \times (0,T)
\end{align}
for $i=1,\dots, d-1$.
In Eqn. \eqref{NSslip1}, $\alpha:= \alpha(x)$ is a smooth positive function. The (variable) slip-length is defined as $\ell_s:= \nu/\alpha$.   This boundary condition allows the fluid to slip tangentially along the boundary for all $\nu>0$.  Both the no-slip and Navier-friction condition above arise rigorously from the Boltzmann equation in the hydrodynamic limit with appropriate scalings \cite{JM17}. The nature of the inviscid limit for the Navier-Stokes system \eqref{NSbintro}--\eqref{nonpen} coupled with either of these physical boundary conditions  \eqref{noslip} or \eqref{NSslip1} and its connection to the Euler equations for an inviscid fluid is an outstanding open problem.  
We briefly review the status presently.

   The main physical process which makes the behavior of fluids with small viscosity  so rich and difficult to analyze is the formation of thin viscous boundary layers which may become singular in the inviscid limit, detach from the walls and generate turbulence in the bulk.  In contrast to the situation without solid boundaries, process can occur even if a strong Euler solutions exists (which holds true globally in time, for example, in two spatial dimensions from smooth initial conditions).
A fundamental result in this area is due to Kato \cite{K72}, who proved that the following two conditions are equivalent: (i) the integrated energy dissipation 
vanishes in a very thin boundary layer of thickness $O(\nu)$ and
(ii) any Navier-Stokes solution with no slip boundary conditions at the wall converges strongly in $L^\infty_tL_x^2$ to the
Euler solution as $\nu\to 0$. Additionally, the above holds if and only if  the \emph{global} dissipation $\langle \varepsilon^\nu \rangle$ vanish in the inviscid limit
\be
\langle \varepsilon^\nu \rangle \to 0 \qquad \text{as}\qquad \nu\to 0.
\ee
Another important equivalence condition of particular relevance to our work was established by Bardos and Titi (Theorem 4.1 of \cite{BardosTiti13}, Theorem 10.1 of \cite{Kelliher17}), who prove that convergence to strong Euler in the energy space is equivalent 
to the wall-friction velocity $u_*$ (related to the local shear stress at the wall)\footnote{We recall this convention here to connect with the literature on wall-bounded turbulence.  Obviously, the right-hand-side of \eqref{frictionvel} need not be sign definite.  However, this definition of the friction velocity via the formula \eqref{frictionvel} is borrowed from the turbulent channel flow literature in which $u_*^2:=\nu \partial_2 \ol{u}_1|_{x_2=0}$ where $\ol{u}_1:=\ol{u}_1(x_2)$ is the mean (e.g. Reynolds averaged) velocity profile.  There, it is expect that $\ol{u}_1(x_2)$ is an increasing function near the wall at $x_2=0$, so $\partial_2 \ol{u}_1|_{x_2=0}>0$ and the definition makes sense. } vanishing 
\be\label{frictionvel}
 u_*^2:= \nu   (\partial_{\hat{n}} u^\nu)_{\hat{\tau}}  \wc 0 \qquad \text{as}\qquad \nu\to 0
\ee
in a weak sense on $\partial \Omega\times [0,T]$, integrating against $\varphi\in C^1([0,T]\times \partial\Omega)$ test functions.

 Aside from these equivalence theorems, most of the known results establish the strong inviscid limit under a variety of conditions, see e.g. ~\cite{kato,BardosTiti13, ceiv,ConstantinKukavicaVicol15,Kelliher07,TemamWang97b,Kelliher06}.    For no-slip boundaries, some unconditional results are known in settings for which laminar flow can be controlled for short times \cite{SammartinoCaflisch98b,Maekawa14,GV16,Kelliher09,LopesMazzucatoLopesTaylor08,MazzucatoTaylor08}. These unconditional results hold before any boundary layer separation or other turbulent behavior occurs.  
 
 On the negative side, it has been shown recently that the Prandtl Ansatz is, in general, false for no-slip conditions and that the $L^\infty$-based Prandtl expansion fails for unsteady flows \cite{GN18}.  Moreover, there is a vast amount of experimental and numerical evidence for anomalous dissipation, i.e. the phenomenon of non-vanishing dissipation of energy in the limit of zero viscosity, in the presence of solid boundaries.
 For example, see the experimental work of \cite{KRS84,PKW02} from wind tunnel experiments and of \cite{C97} for more complex geometries.  In two-dimensions, the works \cite{Farge18} and \cite{Farge11}  convincingly show through a careful numerical study that anomalous dissipation occurs from vortex dipole initial configurations with both no-slip \eqref{noslip} and Navier-friction conditions \eqref{NSslip1} respectively.  See extended discussion of the evidence in \cite{DN18,DE17leray}.  In light of Kato's equivalence,  in situations exhibiting anomalous dissipation convergence cannot be towards a strong solution of Euler. Recently progress has been made towards giving minimal conditions for the inviscid limit to weak Euler solutions to hold \cite{CV17, DN18wl,CLNLV18}.  Such solutions may provide a framework to describe the anomalous dissipation in the inviscid limit as envisioned by Onsager \cite{O49}.  See \cite{DN18,DN18wl,BT18,BTW18} for recent progress in this direction.

The purpose of the above review is to provide sharp relief to the results of the present paper which are summarized in the main Theorem. According to our theory (at least in two-dimensions), when polymers are attached to the wall,
these additives provide a mollifying effect in the inviscid limit which allows passage of solutions to strong solutions of the Euler equations.   Moreover, we obtain a precise bound, $O(\Re^{-1})$, on the rate that the wall-friction and global dissipation vanishes.\footnote{ The mechanism by which polymer reduces drag is -- effectively -- to create an slip-length at the wall which is constant in $\Re$. However, physically, this prediction should be interpreted as an intermediate-asymptotics for large but finite $\Re$.  Specifically, the $\Re$ regime in which our prediction holds is restricted by the assumptions which lead the the derivation of our PDE system.  The most restrictive of these is the assumptions that, from the macroscopic point of view, the polymers form a continuous carpet at the walls.  As $\Re$ increases without bound, small eddies containing an appreciable amount of energy will develop down to the typical length-scale $R$ of the polymer and therefore likely invalidate this particular (and possibly other) assumption. We will revisit these issues in Remarks \ref{validity} and  \ref{restinvlim}.}  As discussed above, the vanishing of both these objects for no-slip boundary conditions are necessary and sufficient for passage to strong Euler in the inviscid limit, although neither case yet been unconditionally proved for arbitrary finite times in that setting and there is strong evidence that in general such convergence will fail \cite{Farge18,Farge11}.
 In terms of the friction factor $f$, our rate agrees qualitatively with the Hagen--Poiseuille law $\langle f\rangle = 64/\Re$ which can be observed experimentally in laminar pipe flow  \cite{Virk,BA87a}.  
Thus, our prediction is that the introduction of end-functionalized polymer effectively laminarizes the flow.    

In summary, we have formulated a macroscopic model to study fluid-polymer interaction where the effect of the polymer is confined to the wall and is mathematically described by a dynamic boundary condition.  Furthermore, we prove that the resulting equations form a well-posed initial value problem and exhibit drag reduction in a quantitative way.  We believe that our model can shed some light on the essential physical mechanisms behind the polymer drag reduction phenomenon. 
In particular, it has been argued that the observed drag reduction phenomenon requires that the wall-normal vorticity flux  be drastically reduced by polymer additives \cite{Eyink08}, and some possible mechanisms for this reduction are therein discussed.  It would be interesting to use our model to explore and clarify the relevant mechanisms.

\section{Navier-Stokes -- End-Functionalized Polymer System}\label{derivation}

Here, we provide a \emph{formal} (non-rigorous) derivation of a system of equations and boundary conditions to describe the setting of a neutral fluid confined to a domain with end-functionalized polymer along the solid walls.  Our assumptions, $(A_1)$--$(A_8)$, are detailed below.

\subsection{Kinetic Theoretic Derivation}
We consider general bounded domains $\Omega \subset \mathbb{R}^d$ for $d\geq 2$ with smooth boundary $\partial\Omega$. At the end of the section, we will discuss the interpretation for two-dimensional case. Our models are based on the following set of assumptions  (see Figure \ref{fig} for a schematic multi-scale cartoon): \vspace{2mm}
\begin{enumerate}[label=(\subscript{A}{{\arabic*}})]
\item  \emph{One-end anchored.}   The layer consists of polymers floating in the solvent with one end anchored to the wall (e.g. chemically bound or strongly adsorbed).  \vspace{2mm}
\item \emph{Wall coating.}  The grafted polymers covers the boundary surface, and the thickness of this covering layer is the order of characteristic length-scale, denoted by $R$, of polymers. We can think of $R$ as the gyration radius of the tethered polymer.  \vspace{2mm}
\item \emph{Multi-scale assumption.}  We assume that at the scale of the polymer, the surrounding fluid can be described as a continuum and also that the polymer appears `infinitesimal' from the perspective of the macroscopic fluid, i.e. we assume scale seperation
\be \label{Multiscale}
\lambda_{mf} \ll R\ll \lambda_\nabla,
\ee 
where $\lambda_{mf}$ is the mean-free path of the molecules making up the solvent and $\lambda_\nabla$ is the gradient length of the continuum description of the fluid (i.e. typical variation scale of the macroscopic flow).  In particular, the polymer should fit well within the near-wall viscous sublayer of the flow.  Additionally, in the case of domains with curvilinear boundary, we  assume that the typical scale of the polymer $R$ is much small relative to the radius of curvature of the boundary
\be
R\ll \text{(minimum radius of boundary curvature)},
\ee
say ${1}/{R} > 4 \max_{x\in\partial\Omega}\kappa$, where $\kappa$ is the boundary curvature defined by \eqref{curvature}. Therefore, the configuration space for polymers at $x \in \partial \Omega$ with its outward normal vector $\hat{n} = \hat{n} (x)$ is given by a flat half-space,
\be\label{Mx}
M(x) :=  \{ m \in \mathbb{R}^{d} : m \cdot (- \hat{n}(x)) > 0 \}.
\ee 
In the case where finite extend mode is employed (e.g. FENE), then this domain is intersected with a ball $B_r(0)$, thereby building in the finite stretching range $r$ of the polymer.\vspace{2mm}
\end{enumerate}

The above assumptions are concerned with small-scale polymer structure and allow us to determine how the polymer `sees' the large scale fluid solvent and the boundary.  We now make an assumption on the structure of the near-wall velocity at those scales of $O(R)$, which determines how the fluid interacts with the polymers.  This ``microscopic" structure assumption will be forgotten in our continuum model, within which it translates simply to a tangential slip velocity along the boundary.\begin{enumerate}[label=(\subscript{A}{{\arabic*}})]
  \setcounter{enumi}{4}
\item \emph{Velocity field of the flow inside the layer.} Microscopically (at the scale of the polymer $R$), we approximate the velocity of the flow inside the layer by a linear shear.  Specifically, the velocity linearly interpolates between the wall side where it vanishes (assuming no-slip on the polymer scale) and its value at near the boundary of the polymer layer which is $u$ and which is tangent to the boundary.   See third panel of Fig \ref{fig}. This ``outer" velocity $u$ becomes the velocity \emph{at the boundary} in our macroscopic closure.  
\end{enumerate}

\begin{figure}[h!]
\centering
\includegraphics[width=1 \textwidth]{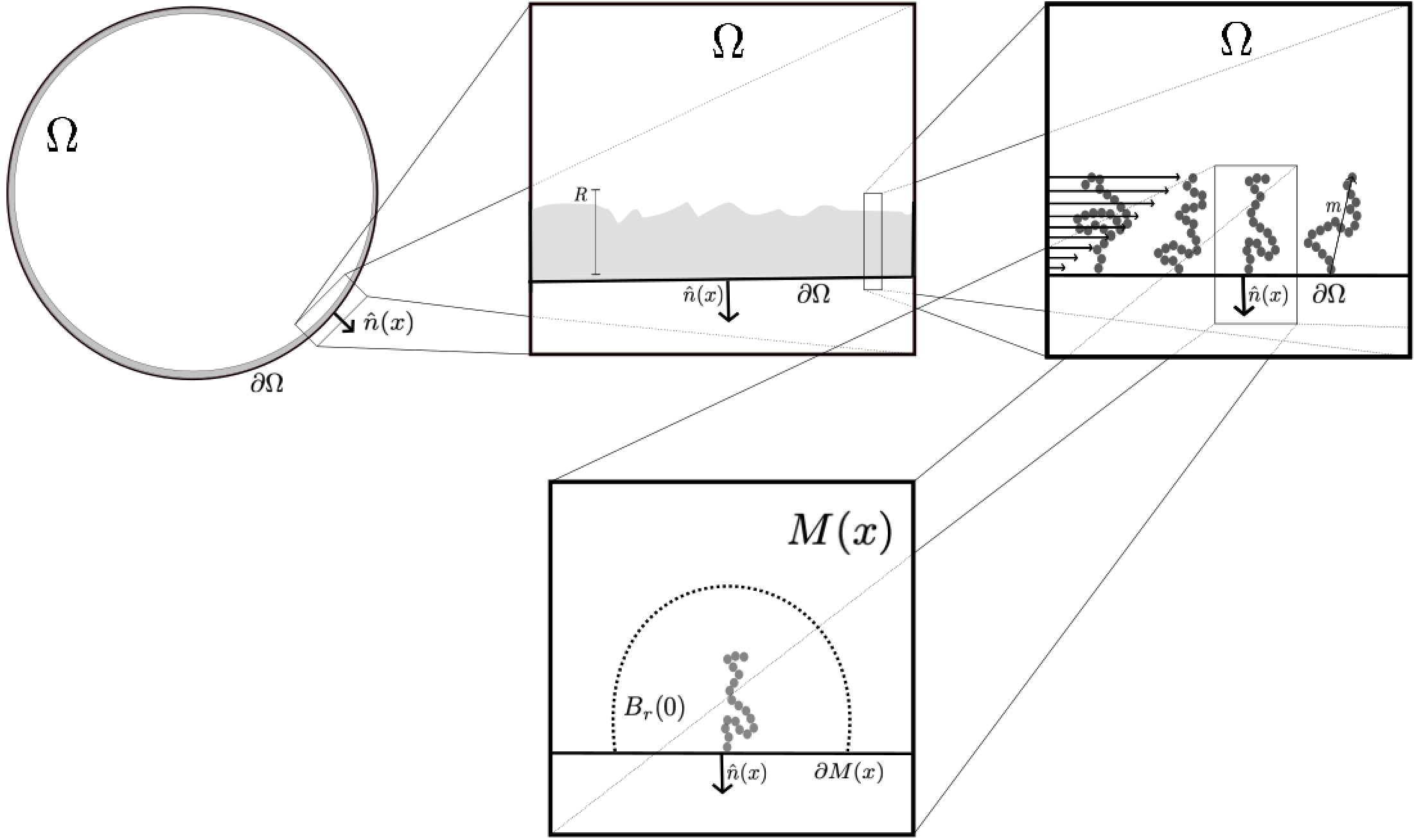}
\caption{
Schematic of the basic multi-scale nature of our polymer model.
}\label{fig}
\end{figure}

Because of assumptions $(A_1)$ -- $(A_4)$, we impose the following boundary condition: since the thickness of the layer is far less than the macroscopic length-scale, we only care about the response of the layer for the flow at wall. We do not incorporate the thickness or shape of the layer in our model. 
We do not have stress balance condition for normal stress $\hat{n} \cdot \Sigma_{F} \cdot \hat{n}$.\footnote{One can ask whether or not the normal stresses also balance, i.e. whether $\hat{n} \cdot \Sigma_L \cdot \hat{n} = \hat{n} \cdot \Sigma_F \cdot \hat{n}$. In our work, we work in a regime in which the layer does not appreciably move or deform in the normal direction.  Consequently, the net force (per unit area) in the normal direction acting on the layer is zero, that is, $  \Sigma_L \cdot \hat{n} + \vec{N} =  \Sigma_F \cdot \hat{n}$, where $\vec{N}$ is the normal force (per unit area) that the wall exerts to the polymer layer. That is, the fluid parcels adjacent to the wall feel the presence of the wall in the normal direction. To explain further, we note that along the fluid-layer boundary the force (per unit area) $(\Sigma_F - \Sigma_L) \cdot \hat{n}$ is applied to the layer. On the other hand, along the layer-wall boundary the normal force (per unit area) $\vec{N}$ is applied to the layer. Then we have balance of two forces, as the layer is steady in the normal direction.} On the other hand we have stress balance condition for the shear stress since the layer, which is a mixture of solvent and polymer, covers the wall.
We formalize this as an assumption:

\begin{enumerate}[label=(\subscript{A}{{\arabic*}})]
  \setcounter{enumi}{4}
\vspace{2mm}
\item  \emph{Tangential stress balance.}  The layer along (impermeable) wall exerts elastic stress due to the restoring force of the fluid-polymer layer which balances the viscous stress of the bulk fluid.  \vspace{2mm}
  \end{enumerate}

  This assumption gives the following: given a point $x$ on the boundary, let $\hat{n}$ be the outward normal vector and $u$ be the fluid velocity at $x$. Let $\Sigma_{L}$ be the stress exerted by the layer {(normalized by  $\rho$)}, and $\Sigma_{F}$ be the stress exerted by the bulk fluid. By impermeability and $(A_5)$ we have
\begin{align}
u \cdot \hat{n} &= 0, \qquad\qquad\quad \quad \ \ \ \ \text{on}\quad \partial \Omega,\\
\hat{\tau}_i \cdot \Sigma_{L} \cdot \hat{n} &=  \hat{\tau}_i \cdot \Sigma_{F} \cdot \hat{n}, \qquad\quad \ \text{on}\quad  \partial \Omega,
\label{Macro}\\
&\qquad\qquad\qquad\qquad \quad  \ i=1,  \dots, d-1,\nonumber
\end{align}
where, for every $x\in \p\Omega$, the vectors $\{ \hat{\tau}_i(x)\}_{i=1}^{d-1}$ form an orthogonal basis of the tangent space of $\partial\Omega$ at $x$.   The stress that the layer exerts is 
a combination of that due to polymer $\Sigma_{P}$ and fluid solvent $\Sigma_{S}$ in the layer,
\be
\Sigma_{L} = \Sigma_{S} + \Sigma_{P}.
\ee
The stress associated to the solvent in the layer is determined from assumption $(A_4)$.  In particular, it is set by the relative velocity near the wall (as it is in for, e.g. Navier-friction boundary condition) so that  $\hat{n} \cdot \Sigma_{S} = -\frac{\nu}{2R} u + \vec{N}$, where $\vec{N}$ is the wall normal force. 
The corresponding stress balance (\ref{Macro})\footnote{Without polymer, this stress-balance argument yields the Navier-friction boundary condition \eqref{NSslip1}.  Specifically, under the assumption $(A_4)$, we consider a fluid parcel of thickness $\lambda$, which is much smaller than the flow length-scale $L$, which is in contact with the wall. As in our case, we set up an effective boundary condition on top of this fluid parcel. Again we assume there is no inflow from the rest of the fluid domain to this fluid parcel. Then, its normal stress $\Sigma_{L} \cdot \hat{n}$ can be similarly approximated by $ -\frac{\nu}{2\lambda} u$ and by the continuity of stress for a Navier-Stokes fluid we obtain
\be \label{NScondition}
2 \left ( D({u} ) \hat{n} \right ) \cdot \hat{\tau}_i + \frac{1}{2 \lambda/L} {u} \cdot \hat{\tau}_i = 0.
\ee
The natural regime of validity for the above assumptions to hold in a viscous fluid without polymer additives forces $\lambda=O(\nu)$ so that the layer lies within the viscous sublayer.  In this way, \eqref{NScondition} recovers the physical Navier-friction boundary condition \eqref{NSslip1} which is rigorously derivable in the hydrodynamic limit from Boltzmann \cite{JM17} (see also \cite{EGKM18}.)}
then reads
\begin{align*}
\hat{n} \cdot \Sigma_F \cdot \hat{\tau}_i &= \hat{n} \cdot \Sigma_P \cdot \hat{\tau}_i - \frac{\nu}{2R} u \cdot \hat{\tau}_i, \qquad\quad \ \text{on}\quad  \partial \Omega,\\
&\qquad\qquad\qquad\qquad\quad\qquad \quad  \ i=1,  \dots, d-1.\nonumber
\end{align*}

The final ingredient for our model is then $\Sigma_{P}$, the polymer layer stress. To obtain this, we need to say something about the structure and dynamics of the polymer additives.  Based on  $(A_1)-(A_4)$, we  assume\vspace{2mm}

\begin{enumerate}[label=(\subscript{A}{{\arabic*}})]
  \setcounter{enumi}{5}
\item \emph{Bead-Spring approximation.} Polymers are modeled as elastic dumbbells whose configuration is characterized by an end-to-end vector $m$ with one end anchored to the wall and the other end free to move. They are taken to have a spring potential $k_B \overline{T} U(m)$, where $U(m)$ is non-dimensional spring potential.  See Figure \ref{fig2}.  \vspace{2mm}

\begin{figure}[h!]
\centering
\includegraphics[width=0.5 \textwidth]{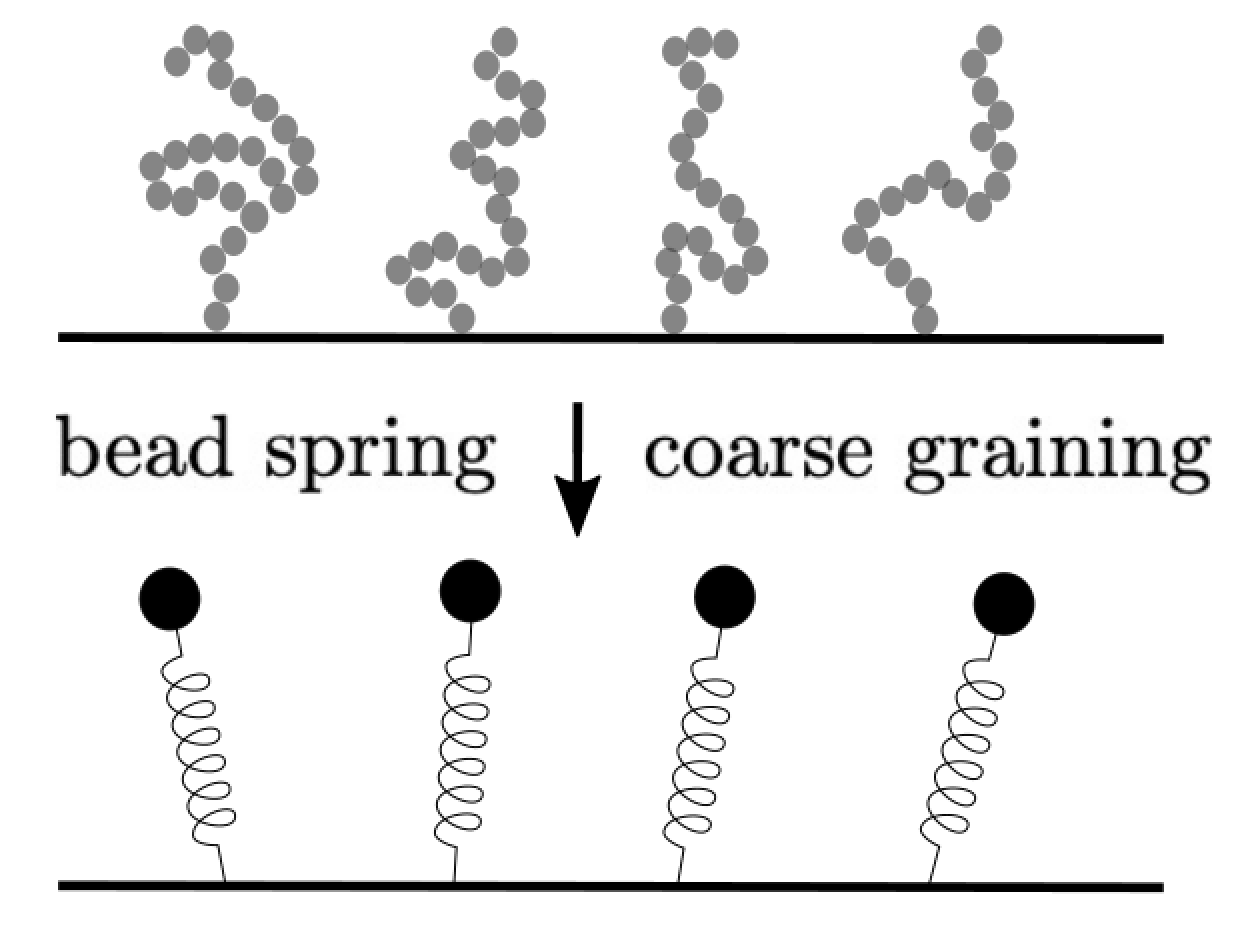}
\caption{
Schematic of the bead spring coarse-graining in configuration space.
}\label{fig2}
\end{figure}

\item \emph{Reflecting condition.} Within the bead-spring approximation, we assume that if the bead hits the boundary then it reflects in the direction of the inward normal vector.  \vspace{2mm}
\item \emph{Single-Chain approximation at the wall.} For simplicity, we ignore the interaction between polymers anchored at the wall. We calculate the dynamics of each polymer as if there is only single chain anchored at the wall, and add them. This puts us in the so-called \emph{mushroom regime}.\footnote{We remark that to be in the ``mushroom regime" in which the polymers do not interact, one requires that the polymer number density $N_P$ defined by \eqref{rhoL} satisfy $N_P <N^*$ where $N^*\sim a_0^{-2} N^{-6/5}$ where $N$ is the polymerization index \cite{dG80} and $a_0$ is the monomer size (see  Chp. 13 of \cite{NA03}).}
 \vspace{2mm}
 
\end{enumerate}

From assumptions  $(A_1)-(A_8)$, we may describe the dynamics of polymers anchored at the wall, and derive Fokker-Planck equation for the polymer probability distribution, denoted by $f_P(x, m, t)$. The final ingredient of the model, required for \eqref{Macro}, is the expression for the stress, and we use Kramers formula \cite{Ottinger12}:
\begin{align} 
\Sigma_{P} = \frac{k_B \overline{T}}{\rho}  \int_{M(x)} m \otimes \nabla_m U f_Pdm.  \label{Kramers}
\end{align}
Although the expression \eqref{Kramers} is standard in theoretical polymer physics, we provide a short derivation in Appendix \ref{AppendKramer} as it is crucial for the derivation of our model. 
We make a brief remark now about dimensions.
 We note that $\rho$, the solvent mass-density, is taken constant and has units of $M/L^d$.  Then ${k_B \overline{T}}/{\rho} $  has units $L^{2+d}/T^2$. Also we assume that polymers are uniformly grafted over the wall.  Specifically, the polymer number density $N_P $ at every $x\in \partial\Omega$ (which is preserved in time by the dynamics for each $x$), is taken to be constant on the boundary, i.e.
\begin{align}\label{rhoL}
N_P := \int_{M(x)} f_Pdm ={\rm (const.)}.
\end{align}
The units of $N_P $ is taken as $1/L^d$. The dimension of ${k_B \overline{T}} N_P /{\rho}$ is $(L/T)^2$, the same as that of stress $\Sigma_L$.

\begin{rem}
Examples for potential choices of configuration spaces and spring potentials are:
\begin{enumerate}
\item Hookean-type dumbbell: we set $r$ in  $(A_4)$ to be $r = \infty$ and 
\be \label{hookeanpot}
U (m) = H \left ( \frac{|m|}{R} \right )^{2k}, \qquad k \ge 1,
\ee
where $H$ is the non-dimensionalized spring constant.  Note that, compared with the standard (dimensional) spring constant $H_{st}$ where $k=1$, we have the relation
$
H_{st}=H {k_B\ol{T}}/{R^2}.$

\item FENE (Finitely Extensible Nonlinear Elastic) models: we have a finite $r < \infty$ in $(A_4)$ and take
\be
U(m) = - H \log \left ( 1 - \frac{|m|^2}{R^2} \right ).
\ee
\end{enumerate}
\end{rem}
 To derive a governing equation for the end-functionalized polymers, we follow Ottinger \cite{Ottinger12}. For the polymer of configuration $m$, anchored at the wall of position $x$ (according to $(A_6)$) and initially in configuration $m_0$, the evolution of $m:=m_t(m_0)$ is determined by the deterministic forces (drift velocity and elastic restoring force) and random fluctuation. 
 Since the length-scale of the polymer $R$ is assumed small relative to the minimum radius of curvature at the boundary across the domain, a polymer pinned at any given $x\in \partial\Omega$ on the boundary is assumed to wander around the half-space $M(x)$ defined by the normal $\hat{n}(x)$ at that point (according to $(A_3)$).  Moreover, we assume that if the polymer end is simply reflected in the direction of the wall-normal $\hat{n}(x_0)$ in the event that it randomly hits the boundary (according to $(A_7)$). 
 
 Specifically, under the bead-spring approximation $(A_6)$,  drift velocity from the near-wall linear shear $(A_4)$ on the polymer is given by
\begin{align}
\text{(drift by fluid experienced by polymer)} =\left ( \frac{m}{R} \cdot \left ( - \hat{n} \right ) \right ) u  . \label{Drift} 
\end{align}
The elastic restoring force is simply  $\frac{k_B \overline{T}}{\zeta} \nabla_m U$ and also contributes to the drift on the bead.  The noise is assumed to be of additive Brownian type with strength $\sqrt{\frac{2k_B \overline{T}}{\zeta} }$.
 Therefore, for each $x\in \partial \Omega$, the polymer end-to-end extension $m_t(m_0):= m_t(m_0;x)\in M(x)$ is a stochastic process described by a reflecting drift-diffusion process  on the half-plane $M(x)$:
\begin{align}\nonumber
\rmd m_t (m_0) &= \left ( \frac{u(x, t) }{R} m_t (m_0)  \cdot \left ( - \hat{n}(x) \right ) - \frac{k_B \overline{T}}{\zeta} \nabla_m U(m_t(m_0))\right ) \rmd t + \sqrt{\frac{2k_B \overline{T}}{\zeta} }  \rmd W_t  + {\hat{n}}(x) \ \rmd \ell_{t} (m_0),\\
 m_t(m_0)|_{t=0}&=m_0 \in M(x)   \label{SDE}
\end{align}
where $W_t$ is a $d$-dimensional standard Brownian motion, and $\ell_t(m_0)$ is the boundary local time density which, for a stochastic polymer end located at some $m\in M(x)$ at time $t$ is the time within the interval $[0,t]$ which is spent near the boundary $\partial M(x)$ per unit distance \cite{SV71,LS84}.  It is formally defined by
\be
\ell_t(m_0) = \int_0^t \delta \left({\rm dist}(m_s (m_0), \partial M(x)\right)  \rmd s.
\ee
See Theorem 2.6 of \cite{BCS04}. We remark that Lions \& Sznitman \cite{LS84} proved existence and uniqueness of stochastic processes as
strong solutions to this ``Skorohod problem" with Lipschitz drifts and sufficient smooth boundaries with regular normal vectors $\hat{n}$. For an extended discussion, see \S 2 of \cite{DE17}. The Fokker-Planck equation associated to the stochastic differential equation (\ref{SDE}) reads
\begin{align}  \label{FP}
\partial_t f_P+ \nabla_m \cdot \left (  \left ( \frac{u(x, t) }{R} \left ( m \cdot \left ( - \hat{n} \right ) \right ) - \frac{k_B \overline{T} }{\zeta} \nabla_m U \right ) f_P\right )&=  \frac{k_B \overline{T}}{\zeta} \Delta_m  f_P\quad \text{in} \quad [0,T]\times M(x), \\ \label{FPbd}
\hat{n}(x) \cdot \nabla_m f_P&= 0 \hspace{20mm} \text{on} \quad [0,T]\times \partial M(x),
\end{align}
 for each $x\in \partial \Omega$.
To sum up, we arrive at the microscopic/macroscopic system
\begin{align}\label{Msys1}
\partial_t u &= \nabla_x \cdot \Sigma_F  + f_b,\quad\hspace{5mm}    \text{in} \quad   \Omega \times (0,T), \\
u|_{t=0} &= u_0   \qquad\qquad \quad\quad\hspace{5mm} \text{on} \quad \!\!  \Omega \times \{t=0\},\\ 
\nabla \cdot u  &= 0 \qquad\qquad\quad\hspace{5mm} \ \ \ \ \ \ \text{in} \quad    \Omega \times [0,T),\\ \label{noflow}
u \cdot \hat{n} &= 0 \qquad\qquad\quad\quad\hspace{5mm}\ \  \text{on} \quad  \!\!\partial \Omega \times [0,T),\\
\hat{\tau}_i \cdot \Sigma_{F} \cdot \hat{n} &=\hat{\tau}_i \cdot \Sigma_{L} \cdot \hat{n}   \qquad\quad\quad \ \text{on}\quad   \partial \Omega \times (0,T),\label{Msys2}\\
 &\hspace{38mm}  i=1,  \dots, d-1 \nonumber
\end{align}
where $f_b$ is a body forcing, the $\Sigma_F$ is the fluid stress tensor, which for a simple Navier-Stokes fluid reads
\be\label{NSstress}
\Sigma_F := - u^\nu\otimes u^\nu - p^\nu \mathbb{I} + 2\nu D(u^\nu),
\ee
recalling that $D(u)= 1/2(\nabla_x u +( \nabla_x u)^t)$ is the symmetric part of the velocity gradient tensor and 
\be
\hat{\tau}_i \cdot \Sigma_{L} \cdot \hat{n} = \hat{n} \cdot \Sigma_P \cdot \hat{\tau}_i - \frac{\nu}{2R} u \cdot \hat{\tau}_i,
\ee 
where the polymer stress $\Sigma_{P}$ is given by the Kramers expression (\ref{Kramers}), which is closed by the Fokker-Planck equation (\ref{FP}) for the polymer distribution at the boundary,  $f_P$ which is supplied with initial conditions $f_P(0)$.  The system \eqref{Msys1}--\eqref{Msys2} \& (\ref{FP})--\eqref{FPbd} comprises our proposed  microscopic-macroscopic system to describe the Navier-Stokes-fluid/end-functionalized polymer interaction.  Note that due to the impermeability condition $u \cdot \hat{n} = 0$ on the boundary the stress that the fluid exerts on the wall is entirely due to viscosity
\be
\hat{\tau}_i \cdot \Sigma_{F} \cdot \hat{n} = 2 \nu \  \hat{\tau}_i \cdot D(u) \cdot \hat{n}.
\ee

\begin{rem}[On the validity of assumptions]\label{validity}

In our opinion, the most subtle of our assumptions are $(A_4)$ and $(A_8)$. First, one may question whether $(A_8)$ (single-chain approximation so that the polymers do not interact with eachother) can be compatible with $(A_2)$ (that, from the macroscopic point of view, the polymer forms a continuous carpet along the boundary). We believe there is a regime of validity where these assumptions coexist, however, even if it is not the case, we interpret $(A_8)$ as a first-hand approximation of the regime in which polymers are close enough to effectively cover the wall but their interactions are not too strong. This interpretation naturally asks a more realistic assumption to replace $(A_8)$. On regime  to consider is that of the ``polymer brush", in which the polymers are spaced close together on the boundary and may strongly interact with each other \cite{NA03,NA01}. It is unclear to us whether or not a fully macroscopic description for this regime will be possible. If not, a coupled microscopic-macroscopic system must be studied to understand the behavior in this regime.

For $(A_4)$, the central issue is the range of parameters which makes linear shear approximation valid. For large enough $\alp$ and small enough $\Re$, the flow will be laminar near the walls and assumption $(A_4)$ should valid. On the other hand, for large $\Re$ the flow will develop small-scales, possibly invalidating the aforementioned justification of $(A_4)$. If this boundary condition regularizes the macroscopic (outside the polymer layer) near-wall flow and it resembles a linear shear, it provides a supporting evidence for $(A_4)$. 
 It would also be interesting to probe $(A_4)$ by microscopic methods, e.g. using molecular dynamics \cite{Karniadakis1, Karniadakis2}.
\end{rem}

\subsection{Energetics:  microscopic/macroscopic balance} 

\begin{prop}
Suitably smooth solutions of \eqref{Msys1}--\eqref{Msys2} satisfy the following global energy balance
\begin{align} \nonumber
\frac{\rmd}{\rmd t} \left ( \frac{1}{2} \int_\Omega |u|^2 dx + \frac{k_B \overline{T}}{\rho} R \mathcal{E} \right ) &= - \int_\Omega \nabla_x u : \Sigma_F dx \\ 
&- \frac{\nu}{2R} \int_{\partial \Omega} |u|^2 dS -\frac{k_B \overline{T}}{\zeta}  \int_{\partial \Omega}\int f_P\left | \nabla_m (\log f_P+ U ) \right | ^2 dm \rmd S.  \label{Energy}
\end{align}
\end{prop}

\begin{proof}
We set the body force $f_b\equiv 0$ for simplicity. The kinetic energy  for  \eqref{Msys1}--\eqref{Msys2} satisfies 
\begin{align}\nonumber
\frac{1}{2} \frac{\rmd}{\rmd t} \int_{\Omega} |u|^2 dx &= - \int_\Omega \nabla_x u : \Sigma_F dx + \int_{\partial \Omega} u \cdot \Sigma_{F} \cdot \hat{n} \rmd S \\
&= - \int_\Omega \nabla_x u : \Sigma_F dx + \sum_{i=1} ^{d-1} \int_{\partial \Omega} u_{\tau_i} \hat{\tau}_i \cdot \Sigma_{P} \cdot \hat{n} \rmd S - \sum_{i=1} ^{d-1}  \frac{\nu}{2R} \int_{\partial \Omega} | u_{\tau_i} |^2 \rmd S  \label{energy1}
\end{align}
where $ u_{\tau_i} = u\cdot \hat{\tau}_i$ and the last identity comes from (\ref{noflow}). Now we calculate the free energy of $f_L$:
\begin{align} \nonumber
\mathcal{E} &= \int_{\partial \Omega} \int_M f_P\log \left ( \frac{f_P}{N_P  e^{-U} } \right ) dm \rmd S \\
&= \int_{\partial \Omega} \int_M f_P\log f_Pdm\rmd S - N_P  \log N_P  |\partial \Omega | +  \int_{\partial \Omega} \int_{M(x)} U f_P\rmd m \rmd S.
\end{align}
A straightforward computation gives the evolution
\begin{align} \nonumber
\frac{\rmd}{\rmd t} \mathcal{E} &= \int_{\partial \Omega} \int_M \nabla_m f_P\cdot \left ( \left ( \frac{u(x, t) }{R} \left ( m \cdot \left ( - \hat{n} \right ) \right ) - \frac{k_B \overline{T}}{\zeta} \nabla_m U \right ) \right ) dm \rmd S \\  \nonumber
& \qquad - \frac{k_B \overline{T}}{\zeta}  \int_{\partial \Omega}\int \frac{|\nabla_m f_P|^2}{f_P} dm \rmd S +  \frac{\rmd}{\rmd t} \int_{\partial \Omega}  \int_{M(x)} U  f_P\rmd m \rmd S \\ \nonumber
&=  \frac{k_B \overline{T}}{\zeta}  \int_{\partial \Omega}\int \Delta_m U f_Pdm \rmd S - \frac{k_B \overline{T}}{\zeta}  \int_{\partial \Omega} \int \frac{|\nabla_m f_P|^2}{f_P} dm \rmd S \\\nonumber
&\qquad +  \sum_{i=1}^{d-1}  \int_{\partial \Omega} \int \partial_{m_{\tau_i}} f_P( m \cdot \hat{n} ) dm \frac{u_{\tau_i}}{R} \rmd S + \frac{k_B \overline{T}}{\zeta }  \int_{\partial \Omega}\int \Delta_m U f_Pdm \rmd S\\ \nonumber
&\qquad -  \int_{\partial \Omega}\int \frac{k_B \overline{T}}{\zeta} |\nabla_m U|^2 f_Pdm \rmd S + \sum_{i=1} ^{d-1} {\frac{\rho}{k_B \overline{T} } } \int_{\partial \Omega} \frac{u_{\tau_i} }{R} \hat{\tau}_i \cdot \Sigma_{P} \cdot (- \hat{n} ) \rmd S \\
&= \sum_{i=1}^{d-1} \frac{\rho}{k_B \overline{T}} 	 \int_{\partial \Omega} \frac{u_{\tau_i} }{R} \hat{\tau}_i \cdot \Sigma_{P} \cdot (- \hat{n} ) \rmd S - \frac{k_B \overline{T}}{\zeta}  \int_{\partial \Omega}\int f_P\left | \nabla_m (\log f_P+ U ) \right | ^2 dm \rmd S. \label{FEevol}
\end{align}
The tangential polymer boundary stress appears in the evolution \eqref{FEevol} of the free energy.  Therefore, we find that the total energy of the system (kinetic energy of the bulk flow together with the free energy of the polymer layer) satisfies the balance \eqref{Energy}.
\end{proof}

Note that  for fluid models satisfying the following energy condition,
\begin{align} \label{energycond}
\int_{\Omega} \nabla_x u : \Sigma_F dx \ge 0,
\end{align}
 the total energy  \eqref{Energy} is non-increasing in time. This condition holds for a simple Navier-Stokes fluid for which $\Sigma_F$  is given by \eqref{NSstress}, provided that the domain has non-positive boundary curvatures. To see this, 
 note that by incompressibility and the no-flow condition  (\ref{noflow}) we have
 \begin{align} \nonumber
\int_{\Omega} \nabla_x u^\nu : \Sigma_F dx &= \nu\int_{\Omega} |\nabla_x u^\nu|^2 dx + \nu\int_{\Omega} \nabla_x u^\nu :  (\nabla_x u^\nu)^t dx\\ \nonumber
& =\nu\int_{\Omega} |\nabla_x u^\nu|^2 dx + \nu \sum_{i=1} ^{d-1} \int_{\partial \Omega}  (u^\nu\cdot \hat{\tau}_i) \partial_{\tau_i} u^\nu \cdot\hat{n}\   \rmd S\\
&=\nu\int_{\Omega} |\nabla_x u^\nu|^2 dx - \sum_{i, j=1} ^{d-1} \nu\int_{\partial \Omega}  (u^\nu\cdot \hat{\tau}_i ) \kappa_{ij} (u^\nu \cdot \hat{\tau}_j) \rmd S \label{dissNS}
\end{align}
where the boundary curvatures were introduced
\be \label{curvature}
\kappa_{ij} =\hat{\tau}_i \cdot \nabla \hat{n} \cdot \hat{\tau}_j.
\ee  
If $\kappa\leq 0$ (negative semidefinite) at all points on the boundary, then energy condition \eqref{energycond} is automatically satisfied (this is true, for example, the canonical setting of flow on a channel with periodic side-walls for which $\kappa\equiv 0$, or in pipe flow for which the curvature is constant and negative). Otherwise, because of the condition $(A_3)$, if ${1}/{R} > 4 \sup_{x\in\partial\Omega}\kappa$ then we have the control of the curvature term.

\subsection{Macroscopic closure: Navier-Stokes fluid and Hookean dumbbell polymer}

If the solvent is taken to be a incompressible Navier-Stokes fluid and the polymer model is taken to be Hookean, that is, the radius $r$ in (\ref{Mx}) is given by $r = \infty$ and the potential $U$ is chosen to be \eqref{hookeanpot} with $k=1$, i.e. $U(m) = H \left ( \frac{|m|}{R} \right )^2$, we arrive at the closed system under some additional mild assumptions detailed below
\begin{align}\label{NSb}
\partial_t u^\nu + u^\nu \cdot \nabla u^\nu  &= - \nabla p^\nu  + \nu \Delta u^\nu  + f_b\quad\quad \ \ \ \    \text{in} \quad   \Omega \times (0,T), \\
u^\nu|_{t=0} &= u_0  \qquad\qquad\qquad \quad\quad\quad \ \ \ \ \  \text{on} \quad \!\!  \Omega \times \{t=0\},\\ 
\nabla \cdot u^\nu  &= 0  \qquad\qquad\qquad \quad\quad\quad\quad \ \ \ \text{in} \quad   \Omega \times [0,T),\\
u^\nu \cdot \hat{n} &= 0  \qquad\qquad\qquad \quad\quad\quad\quad \ \ \  \text{on} \quad \!  \partial \Omega \times [0,T),\\ 
\left(\partial_t  +  \frac{4H k_B {\overline{T}}}{R \zeta } \right)\left(2(D(u^\nu)  \hat{n})\cdot \hat\tau_i  + \frac{1}{2R}u^\nu\cdot \hat{\tau}_i\right)&= - \frac{ {k_B \overline{T}} N_P }{\rho \nu R} u^\nu \cdot \hat\tau_i  \quad\quad\ \ \  \ \ \ \ \  \ \text{on} \quad \!  \partial \Omega \times (0,T), \label{NSf} \\ 
 &\hspace{48mm}  i=1,  \dots, d-1. \nonumber
\end{align}

To derive this fully macroscopic closure \eqref{NSb}--\eqref{NSf}, first note that the Kramers formula (\ref{Kramers}) for the Hookean dumbbell becomes simply
\begin{align*}
\Sigma_P = 2 H \frac{k_B \overline{T}}{\rho}  \int_M \frac{m}{R} \otimes \frac{m}{R} f_Pdm.
\end{align*}
From the Fokker-Planck equation (\ref{FP}), the evolution of $\Sigma_P$ is derived 
\begin{align*}
\partial_t \left (\Sigma_P \right )_{ij} &= 2H\frac{k_B \overline{T}}{\rho} \int_M  \partial_{m_k}  \left ( \frac{m_i m_j} {R^2} \right ) \frac{u_k ^\nu }{R} (m \cdot (-\hat{n} ) ) f_Pdm\\
&\qquad  - 2H\frac{k_B \overline{T}}{\rho} \frac{k_B \overline{T}}{\zeta} \int_M  \partial_{m_k}  \left ( \frac{m_i m_j} {R^2} \right ) 2H \frac{m_k}{R^2} f_Pdm + 2H\frac{k_B \overline{T}}{\rho} \frac{k_B \overline{T}}{\zeta} \int_M \Delta_m \left ( \frac{m_i m_j} {R^2} \right ) f_Pdm \\
&= \left ( \sum_{\ell=1} ^{d-1} \frac{u^\nu \cdot \hat{\tau}_\ell }{R} \left ( \hat{\tau}_\ell \otimes (-\hat{n} ) \Sigma_P+ \Sigma_P (-\hat{n}) \otimes \hat{\tau}_\ell \right ) - \frac{4H}{R^2} \frac{k_B \overline{T}}{\zeta} \Sigma_P + \frac{4 Hk_B \overline{T}}{R^2 \rho} \frac{k_B \overline{T}}{\zeta} N_P  \mathbb{I} \right )_{ij}
\end{align*}
since $u^\nu \cdot \hat{n}  = 0$. Then, contracting with the appropriate boundary normal and tangent vectors, we have
\begin{align}
\partial_t \left ( \hat{\tau}_i \cdot \Sigma_P \cdot (- \hat{n} ) \right ) &= \frac{u^\nu \cdot \hat{\tau}_i }{R} \left ( (-\hat{n} ) \cdot \Sigma_P \cdot (-\hat{n} ) \right ) - \frac{4H}{R^2} \frac{k_B \overline{T}}{\zeta} \left ( \hat{\tau}_i \cdot \Sigma_P \cdot (- \hat{n} ) \right ) , \label{Evtn}\\
\partial_t \left ( (-\hat{n} ) \cdot  \Sigma_P \cdot (-\hat{n} ) \right ) &= - \frac{4H}{R^2} \frac{k_B \overline{T}}{\zeta} \left ( (-\hat{n} ) \cdot \Sigma_P \cdot (-\hat{n} ) \right ) + \frac{4Hk_B \overline{T}}{R^2 \rho} \frac {k_B \overline{T}}{\zeta} N_P . \label{Evnn}
\end{align}
Note that the evolution of $\left ( (-\hat{n} ) \cdot \Sigma_P \cdot (-\hat{n} ) \right )$ completely decouples and does not depend on the tangential velocity.  Further, equation (\ref{Evnn}) shows that at long times it converges to its equilibrium configuration,
\begin{align}\label{eqval}
\left ( (-\hat{n} ) \cdot \Sigma_P \cdot (-\hat{n} ) \right)_{eq} = \frac{k_B \overline{T} }{\rho} N_P .
\end{align}
For simplicity, we assume that $\left ( (-\hat{n} ) \cdot \Sigma_P \cdot (-\hat{n} ) \right )$ already reached at the equilibrium and therefore can be identified with the constant \eqref{eqval}.  This is non-essential for the macroscopic closure. If so, (\ref{Evtn}) becomes
\begin{align}
\partial_t \left ( \hat{\tau}_i \cdot \Sigma_P \cdot (- \hat{n} ) \right ) = \frac{ k_B \overline{T} N_P}{R\rho}    (u^\nu \cdot \hat{\tau}_i)- \frac{4H}{R^2} \frac{k_B \overline{T}}{\zeta} \left ( \hat{\tau}_i \cdot \Sigma_P\cdot (- \hat{n} ) \right ). \label{ntStrEv}
\end{align}
By (\ref{SE}), (\ref{Msys2}) and (\ref{NSstress}), the above is equivalent to the stated boundary condition of (\ref{NSb})--(\ref{NSf}).

\subsection{Non-dimensionalization}\label{nondim}

Defining a characteristic length, $L$ (say the diameter of  the domain $L={\rm diam}(\Omega)$), characteristic velocity $V$  and convective time scale $T = {L}/{V}$, we write introduce dimensionless variables by taking  $u = V \tilde{u}, t = T \tilde{t}, x = L \tilde{x}$.   Note that the polymer relaxation time is $\lambda=\zeta R^2 / 4 H k_B \ol{T}$.  We may now introduce the non-dimensional Reynolds number $\Re$, Weissenberg number $\Wi$,  the relative stress strength $\St$ and the ratio of polymer to domain size $\alp$ as follows 
\be\label{dimnumbs}
\Re =  \frac{VL}{\nu}, \qquad  \Wi = \frac{\lambda}{T}, \qquad  \St=  \frac{\rho V^2}{k_B \overline{T}  N_P },\qquad  \alp =  \frac{L}{R}.
\ee
For definitions of the physical constants, see the introduction. Also we note that $(A_3)$ translates to $\alp > 4 \kappa$.
With these convensions, the equations for the non-dimensional variables in the bulk become
\begin{align} \nonumber
\partial_{\tilde{t} } \tilde{u}^\nu + \tilde{u}^\nu \cdot \nabla_{\tilde{x} } \tilde{u}^\nu &= - \nabla_{\tilde{x}} \tilde{p}^\nu + \frac{1}{\Re} \Delta_{\tilde{x} } \tilde{u}^\nu + \tilde{f}_b, \\ \nonumber
\nabla_{\tilde{x}} \cdot \tilde{u}^\nu &= 0,
\end{align}
and, on the boundary, the following non-dimensionalized boundary condition holds 
\be
\left(\partial_{\tilde{t} } +\frac{1}{\Wi} \right)\left(2\tilde{D}(\tilde{u}^\nu ) \hat{n} \cdot \hat{\tau}_i + \frac{\alp}{2} \tilde{u}^\nu \cdot \hat{\tau}_i\right) = - \frac{\alp \Re }{ \St} \tilde{u}^\nu \cdot \hat{\tau}_i,\qquad  i=1, \dots, d-1,
\ee
thereby reproducing the system \eqref{NSb}-- \eqref{NSf}. 
Note that, an alternative interpretation of the ratio $ \alp \Re/\St$ appearing in the boundary condition is
\be \label{altnondim}
\frac{\alp \Re }{\St}=  \frac{\alp}{\Wi} \frac{\mu_p}{ \mu_s},   \qquad \mu_s=\rho \nu, \qquad \mu_p=N_P  \lambda k_B \ol{T},
\ee
where involving dynamic viscosities of the solvent $\mu_s$ and polymer $\mu_p$.
The polymer viscosity $\mu_p$ is determined from kinetic theory as 
 (number density)$\times$ (polymer relaxation time)$\times k_B \ol{T}$.
The benefit of the non-dimensionalization \eqref{altnondim} is that it allows one to base a Reynolds number on the total viscosity instead of accounting for the change in $\Re$ due to presence of polymers.\footnote{Occasionally, a fourth parameter known as the elasticity $\textbf{E}:=\Wi/\Re$, is sometimes used. It is the ratio of polymer time scale to viscous time scale; it is thought to be more relevant in many cases. See Figure 4 of \cite{G14} for discussion about parameter regimes for drag reduction for dilute  polymers added to the bulk.
}
 For notational simplicity, we hereon drop the tildes and understand all variables to be dimensionless. That is, we write the system as
\begin{align}
\partial_t u^\nu + u^\nu \cdot \nabla u^\nu  &= - \nabla p^\nu  + \frac{1}{\Re} \Delta u^\nu  + f_b\quad\quad    \text{in} \quad   \Omega \times (0,T), \\
u^\nu|_{t=0} &= u_0  \qquad\qquad\qquad \quad\quad\quad \ \ \ \ \  \text{on} \quad \!\!  \Omega \times \{t=0\},\\ 
\nabla \cdot u^\nu  &= 0  \qquad\qquad\qquad \quad\quad\quad\quad \ \ \ \text{in} \quad   \Omega \times [0,T),\\
u^\nu \cdot \hat{n} &= 0  \qquad\qquad\qquad \quad\quad\quad\quad \ \ \  \text{on} \quad \!  \partial \Omega \times [0,T),\\ 
\left(\partial_t  +\frac{1}{\Wi}\right) \left (2(D(u^\nu)  \hat{n})\cdot \hat\tau_i +\frac{\alp}{2} {u^\nu} \cdot \hat{\tau}_i  \right ) &= -\frac{\alp \Re }{ \St} u^\nu \cdot \hat\tau_i  \quad\quad\quad\quad\quad\quad  \text{on} \quad \!  \partial \Omega \times (0,T),\\
 &\hspace{48mm}  i=1,  \dots, d-1 \nonumber
\end{align}

\begin{prop}\label{nsbalance}
Suitably smooth solutions of \eqref{NSb}-- \eqref{NSf} satisfy the following global energy balance
\begin{align} \nonumber
\frac{\rmd}{\rmd t} &\left(\int_\Omega \frac{1}{2} |u^\nu(x,t)|^2 dx + \sum_{i=1} ^{d-1} \frac{\St}{2 \Re^2 \alp}   \int_{\partial \Omega} |(2D (u^\nu)\hat{n} + \frac{\alp}{2} u^\nu) \cdot \hat{\tau}_i|^2 \rmd S \right)\\ \nonumber
&= -   \frac{1}{\Re}  \int_\Omega |\nabla u^\nu(x,t)|^2 dx+ \int_{\Omega} u^\nu \cdot f_b \rmd x - \sum_{i=1} ^{d-1} \frac{\alp}{2 \Re} \int_{\partial \Omega} |u^\nu \cdot \hat{\tau}_i |^2 dS\\
&\qquad  - \sum_{i=1} ^{d-1} \frac{\St}{\Re^2 \alp \Wi}  \int_{\partial \Omega} |(2D (u^\nu)\hat{n} + \frac{\alp}{2} u^\nu) \cdot \hat{\tau}_i|^2 \rmd S  + \sum_{i,j=1} ^{d-1} \frac{1}{\Re} \int_{\partial \Omega} (u^\nu \cdot \hat{\tau}_i) \kappa_{ij} (u^\nu \cdot \hat{\tau}_j)  \rmd S.\label{energy}
\end{align}
where $\kappa_{ij} :=\hat \tau_i \cdot \nabla \hat n\cdot \hat \tau_j$ are the boundary curvatures.
\end{prop}
\begin{proof}
The balance \eqref{energy} follows from \eqref{energy1} together with \eqref{dissNS} and
and from (\ref{ntStrEv}) in the form
\begin{align}
\frac{1}{2}\frac{\rmd}{ \rmd t} \int_{\partial \Omega} \left ( \hat{\tau}_i \cdot \Sigma_P \cdot \hat{n} \right )^2 \rmd S = - \frac{\alp \Re }{\St} \int_{\partial \Omega} (u^\nu \cdot \hat{\tau}_i ) (\hat{\tau}_i \cdot \Sigma_P \cdot \hat{n}) \rmd S - \frac{1}{\Wi} \int_{\partial \Omega} \left ( \hat{\tau}_i \cdot \Sigma_P \cdot \hat{n} \right ) ^2 \rmd S.
\end{align}
Substituting and noting that
$ \hat{\tau}_i \cdot \Sigma_F \cdot \hat{n} = \Re^{-1}(2D (u^\nu)\hat{n}) \cdot \hat\tau_i$
completes the proof.
\end{proof}

\begin{rem}[Navier-Stokes -- End-Functionalized Polymer system in two-dimensions]
Of course, one may always regard the system (\ref{NSbintro})--(\ref{NSfintro}) in $2d$ as simply a mathematical analogue of the $3d$ situation. However, there are physical regimes in which the two-dimensional equations should appear as the correct effective dynamics.  On immediate difficulty in doing so is, as discussed in Footnote 1 of the introduction, the validity of Stokes-Einstein relation (\ref{SE}) in two dimensions is not well established. On the other hand, we argue now that, if the spring potential is Hookean, then we may regard the system (\ref{NSbintro})--(\ref{NSfintro}) in $2d$ as a representation of the fluid-polymer system in $3d$ which is either confined in a large aspect ratio domain or homogeneous in one direction. To understand this,  note that although we think of two-dimensional flow, physically fluids occupy three-dimensional space. If the domain is taken to be $\Omega = \{ (x_1, x_2, x_3) \in \Omega_P \times I \}$, then we argue that the flow is well described by two dimensional dynamics if either (i) $|I|$ is much smaller than the scale of $\Omega_P$, or (ii) $I = \mathbb{T}^1$ and the flow is homogeneous in $x_3$ direction. In the case (i), the multi-scale assumption (\ref{Multiscale}) should be interpreted as that $R$ is also much smaller than the scale of $|I|$. In both cases, (\ref{FP})--(\ref{FPbd}) can be formally rewritten in terms of 
\be \nonumber
f_P^* (x^*, t, m^* ) = \int f_Pdm_3,
\ee
where $x^* = (x_1, x_2)$ and $m^* = (m_1, m_2)$.  Note that $f_L^*$ is independent of $x_3$ since (i) the system already ignores $x_3$ dependence or (ii) the system is homogeneous in $x_3$ direction, by the following:
\begin{align} \nonumber
\partial_t f_P^* + \nabla_{m^*} \cdot \left (   \frac{u(x, t) }{R} \left ( m^* \cdot \left ( - \hat{n} \right ) \right ) f_P^* - \frac{k_B \overline{T} }{\zeta} \int \nabla_{m^*} U  f_Pdm_3 \right )&=  \frac{k_B \overline{T}}{\zeta} \Delta_{m^*}  f_P^* \quad \text{in} \quad [0,T]\times M^*(x), \\ \nonumber
\hat{n}(x) \cdot \nabla_m f_P^* &= 0 \hspace{22mm} \text{on} \quad [0,T]\times \partial M^*(x),
\end{align}
where $M^*(x) = \{ (m_1, m_2) : (m_1, m_2, m_3) \in M(x) \}$, since $u_3 = 0$ and $\hat{n} = (n_1, n_2, 0)$. Crucially, in the Hookean dumbbell case, we have $\nabla_{m^*} U = H m^*$ which is manifestly independent of $m_3$.  Thus,
\be \nonumber
\int \nabla_{m^*} U f_Pdm_3 = \nabla_{m^*} U f_P^*
\ee
and consequently we can replace the boundary equation (\ref{FP})--(\ref{FPbd}) with the above effective $2d$ ones.
\end{rem}

\begin{rem}[Recovery of no-slip boundary conditions]\label{solventLayer}
Note that \eqref{NSf} can be expressed as
\be \nonumber
2 \left (\partial_t + \frac{1}{\Wi} \right ) (2 D(u^\nu) \cdot \hat{n} ) \cdot \hat{\tau}_i + \alp \left ( \partial_t u^\nu + \frac{1}{\Wi} \left ( 1 + \frac{2 \mu_p}{\mu_s } \right )u^\nu \right ) \cdot \hat{\tau}_i = 0.
\ee
If the polymer is taken much smaller than the domain so that the parameter $\alp =L/R$ is taken to infinity with $\Wi$ and $ \frac{\mu_p}{\mu_s }$ fixed, then the formal $\alp \rightarrow \infty$ limit shows that $u^\nu$ converges to the no-slip boundary conditions (if $u_0|_{\partial\Omega}=0$, otherwise they converge exponentially fast (in time) to no-slip).  
\end{rem}

\section{Global existence of strong solutions in $2d$}

It is convenient  for our analysis to express \eqref{NSb}--\eqref{NSf} in terms of the vorticity $\omega=\nabla^\perp \cdot u$ where $\nabla^\perp = (-\partial_2, \partial_1)$.  
By Lemma 2.1 of \cite{Clopeau98}, provided that $u\in H^2(\Omega)$ and $u\cdot \hat{n}=0$ on $\partial\Omega$, then
\be\label{vortident}
 \omega|_{\partial \Omega} = 2(D (u)\hat{n}) \cdot \hat\tau|_{\partial \Omega} + 2\kappa (u\cdot \hat{\tau})|_{\partial \Omega}.
\ee
Thus, the vorticity satisfies the following closed system
\begin{align}\label{vortb}
\partial_t \omega^\nu + u^\nu \cdot \nabla \omega^\nu  &=   \frac{1}{\Re} \Delta \omega^\nu  + \nabla^\perp\cdot f_b   \quad\quad\quad\quad\quad\quad \quad\quad \ \ \ \ \  \text{in} \quad   \Omega \times (0,T),\\
\omega^\nu|_{t=0} &= \omega_0   \qquad\qquad\qquad\qquad\quad\quad\quad\quad\quad\quad  \ \ \  \text{on} \quad   \Omega \times \{0\},\\ 
\label{vortf}
\left(\partial_t  +\frac{1}{\Wi}\right) \omega^\nu &=\left (2\kappa - \frac{\alp}{2} \right ) \partial_t (u^\nu\cdot \hat{\tau}) -\left(\frac{\alp \Re}{ \St} -   \frac{2\kappa - \frac{\alp}{2} }{\Wi}\right) u^\nu \cdot \hat{\tau} \ \ \quad \text{on} \quad  \partial \Omega \times (0,T),
\end{align} 
where, for each fixed time, the velocity $u^\nu$ is recovered from the vorticity using the Biot-Savart law:
\be
u^\nu = K_\Omega[\omega^\nu].
\ee
Here, $K_\Omega$ is an integral operator of order $-1$ with a kernel given by $\nabla^\perp G_\Omega$, where $G_\Omega$ is the Green's function for Laplacian on $\Omega$ with Dirichlet boundary conditions.
More specifically, for any $v \in W^{-1,p}(\Omega)$, the Biot-Savart law says
$K_{\Omega} [v] = \nabla^\perp \psi,$
where $\psi$ is the unique solution of
\begin{align}
\Delta \psi &= v, \quad \quad \!\! \text{in} \quad \Omega,\\
  \psi &= 0 \quad \quad \text{on} \quad \partial \Omega.
\end{align}
By standard elliptic regularity, it follows that for $k\geq 0$ and $p\in (1,\infty)$ if $v\in W^{k,p}(\Omega)$, then $K_\Omega[f]$ satisfies
\be\label{BSbd}
\|K_\Omega[v]\|_{W^{k,p}(\Omega)}\leq C \|v\|_{W^{k-1,p}(\Omega)}.
\ee
For details see e.g. Chapter III \S 4 of  \cite{BF12} and Theorem 1 of \cite{BSBD}. 

\vspace{2mm}

We now prove the following theorem.
\begin{thm}[Global Well-Posedness] \label{thmglobal}
Suppose $\omega_0 \in H^2 (\Omega) \cap C(\bar{\Omega})$. For any $T>0$, there exists a unique 
\be\nonumber
\omega^\nu \in  C(0,T; H^1 (\Omega)) \cap C([0,T]\times\bar{\Omega}) \cap H^1(0,T; L^2 (\Omega))\cap L^2(0,T; H^2(\Omega))
\ee
 solving the system (\ref{vortb}) - (\ref{vortf}) where the boundary condition is understood in the sense of
\begin{align}\label{bcpw}
\omega^\nu (t) = \left (2 \kappa - \frac{\alp}{2} \right )  u^\nu (t) \cdot \hat{\tau} + e^{-\frac{1}{\Wi} t } \left ( \omega_0 - \left (2 \kappa - \frac{\alp}{2} \right )  u_0 \cdot \hat{\tau}  \right ) -\frac{\alp\Re}{\St}  \int_0 ^t e^{-\frac{1}{\Wi} (t-s) } u^\nu (s) \cdot \hat{\tau} ds
\end{align}
holding pointwise in $(t,x)\in[0,T]\times \partial \Omega $.
\end{thm}
For simplicity of notation, we denote $\beta = 2\kappa - \frac{\alp}{2}$.

\subsection{A priori estimates}

First, the energy balance for the Navier-Stokes -- End-Functionalized system immediately gives some apriori control on the kinetic energy and viscous energy dissipation.   We note that this control does not depend on the particular model of the spring potential $U$ used in the model.

\begin{lemma}[Energy Bounds]\label{lem0}
For any $T>0$, we have
\begin{align} \nonumber
&\| u^\nu \|_{L^\infty ( 0. T; L^2({\Omega} )) }^2 + \frac{1}{\Re} \| u^\nu \|_{L^2 (0; T; H^1 (\Omega) ) } ^2 + \frac{\alp}{4 \Re} \| u^\nu \|_{L^2 (0. T; L^2 (\partial \Omega)) } ^2  \\ 
&\le e^T \left ( \| u_0 \|_{L^2 (\Omega) } ^2 + \| f_b \|_{L^2 (0, T; L^2 (\Omega) } ^2 + \frac{\St}{\Re ^2 \alp} \left ( \| 2D (u_0 ^\nu ) \hat{n}|_{L^2 (\partial \Omega) } ^2 + \frac{\alp^2}{4} \| u_0 \|_{L^2 (\partial \Omega ) } ^2 \right ) \right ). \label{energybnd}
\end{align}
\end{lemma}
\begin{proof}
Recall the balance \eqref{energy} with $\alp > 4  \max_{x\in \partial \Omega}  \kappa$, which is consistent with our assumption $(A_3)$. For general spring potential $U$, we start from \eqref{Energy}.
\end{proof}

The system \eqref{NSb}--\eqref{NSf} also admits an apriori estimate for the vorticity in $L^\infty$ spacetime, at least within the Hookean dumbbell closure.  The proof of this fact follows essentially from the argument to prove Lemma 3 of \cite{Filho05}
which holds for Navier-friction boundary conditions.  Remarkably, the $L^\infty$ bound on vorticity is insensitive to high Reynolds number --  this is a consequence of the Stokes-Einstein relation \eqref{SE} for the bead-friction coefficient of the polymer  which is reflected in the ratio $\alp\Re\Wi/\St$ being independent of Reynolds $\Re$ if the latter is varied either by changing solvent viscosity $\nu$ or characteristic velocity $V$. This will be discussed at length in Remark \ref{restinvlim}.

\begin{lemma}[Vorticity Bound]\label{lem1}
For any $T>0$, there exists $C_2>0$ defined by \eqref{unifbound} such that
\be\label{vorticitybnd}
\| \omega^\nu \|_{C ( [0, T] \times \bar{\Omega} ) } \leq C_2 .
\ee
\end{lemma}
\begin{proof}
Let $C_1$ be the right side of \eqref{energybnd}. For any $p > 2$, from the  embedding and Sobolev interpolation between $W^{1,p}$ and $L^2$ we have 
\begin{align}\nonumber
\| u^\nu (t) \cdot \hat{\tau} \|_{L^\infty ( \partial \Omega ) } &\le \| u^\nu(t) \|_{C ( \bar{\Omega} ) } \le \| u^\nu(t) \|_{L^2 (\Omega) } ^\theta \| u^\nu \|_{W^{1, p} (\Omega) } ^{1- \theta} \le C \| u^\nu(t) \|_{L^2 (\Omega) } ^\theta \| \omega^\nu  (t) \|_{L^p (\Omega ) } ^{1 - \theta} \\ \nonumber
&\le C^{\frac{1}{\theta}} \epsilon^{-\frac{1-\theta}{\theta} } \sup_{t\in [0,T]}  \| u^\nu (t) \|_{L^2 (\Omega) } + \epsilon \sup_{t\in [0,T]} \| \omega^\nu  (t) \|_{L^p (\Omega ) } \\
&\le C \sqrt{C_1} \epsilon^{-1} + \epsilon \| \omega^\nu  \|_{L^\infty (0, T; L^\infty( \Omega ))},  \label{ubd}
\end{align}
where $\theta = \frac{p-2}{2(p-1) }$, we used the energy bound from Lemma \ref{lem0} and Young's inequality introduced the arbitrarily small $\epsilon$ and taking the limit $p \rightarrow \infty$. On the other hand, from Duhamel's formula and (\ref{vortf}) we obtain \eqref{bcpw}. Also note that $|\beta | \le \alp$.
Therefore, we have the following
\begin{align} \nonumber
\| \omega^\nu  (t) \|_{L^\infty (\partial \Omega ) } &\le 2\alp \|u^\nu \cdot \hat{\tau} \|_{L^\infty ((0, T) \times \partial \Omega ) }+  \| \omega_0 \|_{L^\infty  (\partial \Omega ) } \\ \nonumber
&\qquad + \frac{\alp \Re}{ \St} \int_0 ^t e^{-\frac{1}{\Wi} (t-s) } \|u^\nu \cdot \hat{\tau} \|_{L^\infty ((0, T) \times \partial \Omega ) } ds \\
&\le \| \omega_0 \|_{L^\infty  (\partial \Omega ) }  + \left(2\alp  + \frac{\alp \Re\Wi}{\St} \right)\|u^\nu \cdot \hat{\tau} \|_{L^\infty ((0, T) \times \partial \Omega ) }\nonumber\\
&\le \left ( 2 \alp + \frac{\alp \Re \Wi}{\St} \right ) \left (C \sqrt{C_1} \epsilon^{-1} + \epsilon \| \omega^\nu  \|_{L^\infty (0, T; L^\infty( \Omega ))} \right ) + \| \omega_0 \|_{L^\infty  (\partial \Omega ) }.
\end{align}
 On the other hand, from maximum principle we have
\begin{align}
\| \omega^\nu  \|_{C ( [0, T] \times \bar{\Omega} ) } \le \| \omega_0 \|_{L^\infty ( \Omega ) } + \| \omega^\nu  \|_{L^\infty ((0, T) \times \partial \Omega ) } + T \| \nabla^\perp \cdot f_b \|_{L^\infty ( [0, T] \times \bar{\Omega} ) }.
\end{align}
By taking $\epsilon$ small enough, 
\begin{align} \nonumber
\epsilon &= \frac{1}{2} \left ( 2 \alp + \frac{\alp \Re \Wi}{\St} \right )^{-1}, \\ 
C_2 &= 4 \left ( 2 \alp + \frac{\alp \Re \Wi}{\St} \right )^2 C \sqrt{C_1} + 4 \| \omega_0 \|_{C ( \bar {\Omega} ) } + 2T \| \nabla^\perp \cdot f_b \|_{L^\infty ([0, T] \times \bar{\Omega} ) }, \label{unifbound}
\end{align}
we may conclude the claimed bound \eqref{vorticitybnd}.
\end{proof}

\begin{lemma}[Higher Regularity]\label{lem2}
For any $T>0$, there exists  $C:=C(\Re,\Wi,\St,\alp,u_0,\Omega,T)$  such that
\be\label{velH2}
\| \omega^\nu \|_{C(0,T; H^1(\Omega))  } \leq C , \qquad  \|\Delta \omega^\nu \|_{L^2(0,T; L^2(\Omega))}\leq C, \qquad \| \omega^\nu \|_{H^1(0,T; L^2(\Omega))  } \leq C.
\ee
\end{lemma}

\begin{proof}
By multiplying $(-\Delta) \omega^\nu$ to (\ref{vortb}) and integrating we have
\begin{align} \label{vortHD}
\int_{\Omega} (-\Delta \omega^\nu) \partial_t \omega^\nu dx + \frac{1}{\Re} \int_{\Omega} |\Delta \omega^\nu |^2 dx = \int_{\Omega} \Delta \omega^\nu u^\nu \cdot \nabla \omega^\nu dx - \int_{\Omega} \Delta \omega^\nu \nabla^\perp \cdot f_b dx.
\end{align}
Note now that the first term of the left hand side of (\ref{vortHD}) can be rewritten as
\begin{align} \label{vortHDRHS}
- \int_{\Omega} \nabla \cdot \left ( \nabla \omega^\nu \partial_t \omega^\nu \right ) dx + \int_{\Omega} \nabla \omega^\nu \cdot \partial_t \nabla \omega^\nu dx = - \int_{\partial \Omega} \hat{n} \cdot \nabla \omega^\nu \partial_t \omega^\nu \rmd S + \frac{1}{2} \frac{\rmd}{\rmd t} \| \nabla \omega^\nu \|_{L^2 (\Omega)} ^2 .
\end{align}
Thus we obtain the following evolution
\begin{align} \nonumber
\frac{1}{2} \frac{\rmd}{\rmd t} \| \nabla \omega^\nu \|_{L^2 (\Omega)} ^2 + \frac{1}{\Re} \|\Delta \omega^\nu \|_{L^2(\Omega)}^2 dx &= \int_{\partial \Omega} \hat{n}\cdot \nabla \omega^\nu \partial_t \omega^\nu  \rmd S\\
&\qquad  + \int_{\Omega} \Delta \omega^\nu u^\nu \cdot \nabla \omega^\nu dx - \int_{\Omega} \Delta \omega^\nu \nabla^\perp \cdot f_b dx.\label{vortHDeqn}
\end{align}
Using the boundary condition (\ref{vortf}) the first term in the right hand side of (\ref{vortHDRHS}) reads
\begin{align} \nonumber
 \int_{\partial \Omega} \hat{n} \cdot \nabla \omega^\nu \partial_t \omega^\nu  \rmd S &= 
\int_{\partial \Omega} \hat{n} \cdot \nabla \omega^\nu  \left(  \beta \partial_tu \cdot \hat{\tau}- \frac{1}{\Wi} \omega^\nu - \left(\frac{\alp \Re}{ \St} - \frac{\beta}{\Wi} \right) u^\nu \cdot \hat{\tau} \right) \rmd S. \nonumber
\end{align}
The second term on the right-hand-side can be written as a bulk term
\be
\frac{1}{\Wi} \int_{\partial \Omega} \hat{n} \cdot \nabla \omega^\nu \omega^\nu \rmd S= \frac{1}{\Wi} \int_{\Omega} \nabla \cdot ( \nabla \omega^\nu \omega^\nu ) dx= \frac{1}{\Wi} \int_{\Omega} \Delta \omega^\nu \omega^\nu  dx+\frac{1}{\Wi} \| \nabla \omega^\nu \|_{L^2 (\Omega) } ^2.
\ee
Therefore, we find that the boundary term becomes
\begin{align} \nonumber
 \int_{\partial \Omega} &\hat{n} \cdot \nabla \omega^\nu \partial_t \omega^\nu \rmd S = -\frac{1}{\Wi} \| \nabla \omega^\nu \|_{L^2 (\Omega) } ^2 -\frac{1}{\Wi} \int_{\Omega} \omega^\nu \Delta \omega^\nu dx\\
&\qquad  - \int_{\partial \Omega} \hat{n} \cdot \nabla \omega^\nu  \left(\frac{\alp \Re}{ \St} - \frac{\beta}{\Wi} \right)u^\nu \cdot \hat{\tau}\   \rmd S +  \int_{\partial \Omega} \hat{n} \cdot \nabla \omega^\nu \beta \partial_t (u^\nu \cdot\hat{\tau}) \rmd S. \label{vortHD3}
\end{align}
The second term of (\ref{vortHD3}) is controlled by 
\begin{align} \label{blterm}
\left|\frac{1}{\Wi} \int_{\Omega} \omega^\nu \Delta \omega^\nu dx\right| \leq\frac{1}{\Wi} \| \Delta \omega^\nu \|_{L^2 (\Omega)} \| \omega^\nu \|_{L^2 (\Omega ) }.
\end{align}
To deal with the third term of (\ref{vortHD3}), we introduce a thin enough tubular neighborhood of $\partial \Omega$, smoothly extend the vector field $\left(\frac{\alp \Re}{\St} - \frac{\beta}{\Wi} \right)\tau$ on $\partial \Omega$ whose support is compactly embedded in this neighborhood, and we denote this vector field as $\Phi$. Then we have 
\begin{align} \nonumber
 \int_{\partial \Omega} \hat{n} \cdot \nabla \omega^\nu \left(\frac{\alp \Re}{ \St} - \frac{\beta}{\Wi} \right) u^\nu \cdot \hat{\tau} \  \rmd S &= \int_{\partial \Omega} \hat{n} \cdot \nabla \omega^\nu u^\nu \cdot \Phi \rmd S \\ \nonumber
&=  \int_{\partial \Omega} \hat{n} \cdot \nabla (\omega^\nu u^\nu \cdot \Phi ) \rmd S -  \int_{\partial \Omega}\hat{n} \cdot   \nabla (u^\nu \cdot \Phi ) \omega^\nu \rmd S \\
&=  \int_{\Omega} \nabla \cdot ( \nabla ( \omega^\nu u^\nu \cdot \Phi ) ) dx -  \int_{\partial \Omega} \hat{n} \cdot \nabla (u^\nu \cdot \Phi )  \omega^\nu \rmd S. \label{vortHD4}
\end{align}
The first term of (\ref{vortHD4}) is controlled by
\begin{align}\nonumber
  \left|\int_{\Omega} \nabla \cdot ( \nabla ( \omega^\nu u^\nu \cdot \Phi ) ) dx\right| &\leq  \left ( \| \Delta \omega^\nu \|_{L^2 (\Omega )} \| u^\nu \|_{L^2 (\Omega ) } \| \Phi \|_{L^\infty (\Omega ) }  + \| \omega^\nu \|_{H^1 (\Omega)} \| u^\nu \|_{H^2 (\Omega) } \|\Phi \|_{W^{1, \infty} (\Omega) }  \right ) \\
&\le c  \|\Delta \omega^\nu \|_{L^2 (\Omega) } \| u^\nu \|_{L^2 (\Omega) }  +C   \|\omega^\nu \|_{H^1 (\Omega) } ^2,\label{flterm}
\end{align}
since $\Phi$ depends only on $\frac{\alp \Re}{\St}, \Wi, \alp$, and $\Omega$ (in particular, on $\kappa$). The second term of (\ref{vortHD4}) is controlled by 
\begin{align}\nonumber
\left| \int_{\partial \Omega} \hat{n} \cdot \nabla (u^\nu \cdot \Phi )  \omega^\nu \rmd S\right| &\leq 
 \| \nabla (u^\nu \cdot \Phi)  \|_{L^2 (\partial \Omega) } \| \omega^\nu \|_{L^2 (\partial \Omega)}\\
& \le  \| u^\nu \cdot \Phi \|_{H^{\frac{3}{2} } (\Omega ) } \| \omega^\nu \|_{H^{1 } (\Omega ) }  \le {C}  \| \omega^\nu \|_{H^1 (\Omega)} ^2\label{flterm2}
\end{align}
by the Sobolev trace inequality. 
It suffices to treat the term
\begin{align} \nonumber
\int_{\partial \Omega} \hat{n} \cdot \nabla \omega^\nu (2\kappa ) \partial_t u^\nu \cdot \hat{\tau} \rmd S.
\end{align}
First note that, from the vorticity equation and the Biot-Savart law, we may express
\begin{align}
\partial_t u^\nu = K_{\Omega} [ \partial_t \omega^\nu] = K_{\Omega} \left[ - \nabla \cdot ( u^\nu \omega^\nu ) + \frac{1}{\Re} \Delta \omega^\nu \right].
\end{align}
Using this correspondence, we have
\begin{align} \nonumber
\int_{\partial \Omega} \hat{n} \cdot \nabla \omega^\nu \beta \partial_t u^\nu \cdot \hat{\tau} \rmd S = \int_{\Omega} \nabla \cdot \left ( \nabla \omega^\nu \Psi \cdot \left ( K_{\Omega} [-\nabla \cdot (u^\nu \omega^\nu ) ] + \frac{1}{\Re}K_{\Omega} [  \Delta \omega^\nu] \right ) \right )dx,
\end{align}
where $T_{\partial \Omega} \Psi = \beta \hat{\tau} $. We now note that
\begin{align} \nonumber
&\left|\int_{\Omega} \nabla \cdot \left ( \nabla \omega^\nu \Psi \cdot  \frac{1}{\Re}K_{\Omega} [ \Delta \omega^\nu]  \right )dx\right|\\ \nonumber
&\quad \leq  \left|\int_{\Omega} \Delta \omega^\nu \Psi \cdot \frac{1}{\Re}K_{\Omega} [ \Delta \omega^\nu] dx\right| + \left|\int_{\Omega} \nabla \omega^\nu \nabla \Psi \cdot \frac{1}{\Re}K_{\Omega} [ \Delta \omega^\nu] dx\right| + \left|\int_{\Omega} \nabla \omega^\nu \Psi \cdot \nabla \frac{1}{\Re}K_{\Omega} [ \Delta \omega^\nu] dx\right| \\ \nonumber
&\quad\le \frac{1}{\Re} \| \Delta \omega^\nu \|_{L^2} \| \Psi \|_{L^\infty} \| K_{\Omega} [\Delta \omega^\nu ] \|_{L^2} + \frac{1}{\Re} \| \nabla \omega^\nu \|_{L^2} \| \Psi \|_{W^{1, \infty} } \| K_{\Omega} [\Delta \omega^\nu ] \|_{H^1} \\ \nonumber
&\quad\le \frac{C(\Psi)}{\Re}  \left ( \| \Delta \omega^\nu \|_{L^2 } \| \Delta \omega^\nu \|_{H^{-1} } + \| \nabla \omega^\nu \|_{L^2} \| \Delta \omega^\nu \|_{L^2 } \right ) \le \frac{C(\Psi)}{\Re}  \| \Delta \omega^\nu \|_{L^2} \| \nabla \omega^\nu \|_{L^2},
\end{align}
where we used $\| \Delta \omega^\nu \|_{H^{-1} } \le \| \nabla \omega^\nu \|_{L^2}$. Note that this estimate involves no boundary terms since $H^{-1}(\Omega)$ is the dual of $H_0^1(\Omega)$. Now, the first term becomes 
\begin{align} \nonumber
&\left|\int_{\Omega} \nabla \cdot \left ( \nabla \omega^\nu \Psi \cdot  K_{\Omega} [-\nabla \cdot (u^\nu \omega^\nu ) ]   \right )dx\right|\\ \nonumber
&\qquad \qquad \leq  \left|\int_{\Omega} \Delta \omega^\nu \Psi \cdot K_{\Omega} [ -\nabla \cdot (u^\nu \omega^\nu ) ] dx\right| + \left|\int_{\Omega} \nabla \omega^\nu \cdot \nabla (\Psi \cdot K_{\Omega} [-\nabla \cdot (u^\nu \omega^\nu ) ] ) dx\right| \\ \nonumber
&\qquad \qquad\le \| \Delta \omega^\nu \|_{L^2} \| \Psi \|_{L^\infty } \| \nabla \cdot (u^\nu \omega^\nu ) \|_{H^{-1}} + \| \nabla \omega^\nu \|_{L^2} \| \Psi \|_{W^{1, \infty }} \| \nabla \cdot (u^\nu \omega^\nu ) \|_{L^2 }\\
&\qquad \qquad\le C \| \Delta \omega^\nu \|_{L^2}  + C' \| \nabla \omega^\nu \|_{L^2}^2
\end{align}
for some constants $C, C'>0$.  To obtain the above, we noted that we used the bounds on $\|u^\nu\|_C, \| \omega^\nu \|_C$ and  $\| \omega^\nu \|_{L^2}$ and therefore $
\| u ^\nu\omega^\nu \|_{H^1} \le   \| u^\nu\|_{H^1 \cap C } \| \omega^\nu \|_{H^1 \cap C}$.  Thus we obtained
\begin{align} \nonumber
\left|\int_{\partial \Omega} \hat{n} \cdot \nabla \omega^\nu (2\kappa ) \partial_t u^\nu \cdot \hat{\tau} \rmd S\right| &\leq \frac{ C(\Psi)}{\Re} \| \Delta \omega^\nu \|_{L^2} \| \nabla \omega^\nu \|_{L^2}+ C \| \Delta \omega^\nu \|_{L^2}  + C' \| \nabla \omega^\nu \|_{L^2}^2\\
& \leq C+  \frac{1}{2 \Re}\| \Delta \omega^\nu \|_{L^2}^2 + C\| \omega^\nu\|_{H^1}^2.\label{ptterm}
\end{align}

Finally, combining  \eqref{blterm}, \eqref{flterm}, \eqref{flterm2} and \eqref{ptterm}, we bound the terms on the right-hand-side of Eqn. \eqref{vortHDeqn} by
\begin{align*}\left| \int_{\partial \Omega} \hat{n}\cdot \nabla \omega^\nu \partial_t \omega^\nu  \rmd S\right|&\leq C+  \frac{1}{2\Re}\| \Delta \omega^\nu \|_{L^2}^2 + C\| \omega^\nu\|_{H^1}^2 + \Re^2 \|\omega^\nu\|_{L^2}^2+  \Re^2 \|u^\nu\|_{L^2}^2, \\\nonumber
\left| \int_{\Omega} \Delta \omega^\nu u^\nu \cdot \nabla \omega^\nu dx - \int_{\Omega} \Delta \omega^\nu \nabla^\perp \cdot f_b dx\right|&\leq \| \Delta \omega^\nu \|_{L^2 } \left ( \|u^\nu \|_{L^\infty } \| \nabla \omega^\nu \|_{L^2 } + \| \nabla^\perp \cdot f_b \|_{L^2 } \right )\\
&\leq   \frac{1}{2\Re}\| \Delta \omega^\nu \|_{L^2}^2 + C\Re \|u^\nu \|_{L^\infty (\Omega)} ^2 \| \omega^\nu\|_{H^1}^2+  \| \nabla^\perp \cdot f_b \|_{L^2 }.
\end{align*}
Noting that by Poincare inequality $\| \omega^\nu \|_{H^1 (\Omega) }$ and $\| \nabla \omega^\nu \|_{L^2 (\Omega)}$ are comparable, and using Cauchy-Schwarz inequality to bury all $\| \Delta \omega^\nu \|_{L^2 (\Omega) }$ terms, we end up with
\begin{align}\nonumber
&\frac{\rmd}{\rmd t} \| \nabla \omega^\nu \|_{L^2 (\Omega)} ^2 + \frac{1}{\Re} \| \Delta \omega^\nu \|_{L^2 (\Omega ) } ^2 + \frac{2}{\Wi} \| \nabla \omega^\nu \|_{L^2 (\Omega)} ^2
  \\ \nonumber
&\quad \leq C(\Re, \Wi,\St, \alp, \Omega) \left ( \left (\| u^\nu \|_{L^\infty (\Omega ) } ^2 + 1\right ) \| \nabla \omega^\nu \|_{L^2 (\Omega ) } ^2\left ( \| \nabla^\perp \cdot f_b \|_{L^2 (\Omega ) } ^2 + \|\omega^\nu \|_{L^2 (\Omega ) } ^2 + \| u^\nu \|_{L^2 (\Omega ) } ^2 \right ) \right ).
\end{align}
Note finally that from the apriori estimate  $\omega^\nu \in C([0,T]\times \ol{\Omega})$ of Lemma \ref{lem1}, we have $u^\nu = K_\Omega[\omega^\nu]\in L^\infty(0,T; W^{1,p}(\Omega))$ for all $1\leq p<\infty$.  In particular, combining this with \eqref{ubd} we find $u^\nu\in C([0,T]\times \ol{\Omega})$.  Whence, by  Lemma \ref{lem1},   the above estimate allows us to conclude that $\omega^\nu \in C(0,T; H^1(\Omega))$ and consequently  $u^\nu \in C(0,T; H^2(\Omega))$.  Moreover, from the vorticity equation we have
\be
\|\partial_t \omega^\nu\|_{L^2} \leq \|u^\nu\|_{L^\infty} \| \nabla \omega^\nu\|_{L^2} + \|\Delta \omega^\nu\|_{L^2},
\ee
which implies that $\omega^\nu \in H^1(0,T; L^2(\Omega))$.
\end{proof}

\subsection{Proof of Theorem \ref{thmglobal}: Global Strong Solutions}

To construct the solution for the system (\ref{vortb})-(\ref{vortf}), we first propose the function space for the solution;
\begin{align} \label{Fnsp}
\mathcal{X} &= \{ \omega \in C_t H^1 (\Omega) \cap C_t C (\bar{\Omega}) \cap H_t ^1 L^2 (\Omega) \ | \  \omega(0) \in H^1 (\Omega) \cap C (\bar{\Omega} ), \Delta \omega (0) \in L^2 (\Omega)  \}, \\
\mathcal{X}' &= \{ \omega \in C_t H^1 (\Omega)  \cap H_t ^1 L^2 (\Omega) \ | \  \omega(0) \in H^1 (\Omega) , \Delta \omega (0) \in L^2 (\Omega)  \},
\end{align} 
with the natural norm $ \| \omega \|_{\mathcal{X}} =   \| \omega \|_{C_t H^1 (\Omega ) } + \| \omega \|_{C_t C (\bar{\Omega} ) } + \|  \omega \|_{H^1 _t L^2 (\Omega ) }  $ and $ \| \omega \|_{\mathcal{X}'} =   \| \omega \|_{C_t H^1 (\Omega ) } + \|  \omega \|_{H^1 _t L^2 (\Omega ) }  $. Here $C_t, H_t ^1, L^2_t$ are shorthand for time interval $[0, T]$. To prove Theorem \ref{thmglobal}, we will:
\begin{enumerate}
\item Establish a contraction mapping $F$ in $\mathcal{X}'$, so that for $\omega(0) \in H^1( \Omega) \cap \{ \Delta \omega(0) \in L^2 (\Omega ) \}$, there is unique $\omega \in \mathcal{X}'$ such that $\omega = F(\omega)$ for a short time $T$. \vspace{2mm}
\item Check that if $\omega(0) \in C(\bar{\Omega})$ then $\omega \in \mathcal{X}$ in fact. Then Lemma \ref{lem1} and consequently Lemma \ref{lem2} become valid, establishing a priori estimates on $\mathcal{X}$.\vspace{2mm}
\item Noting that $\Delta \omega (t) \in L^2 (\Omega)$ for almost every $t \in [0, T]$, so we can continue a point close to $T$, thereby obtaining global well-posedness.
\end{enumerate}

\begin{proof}
For the description of boundary behavior, we define the following operator:
\be\label{bdryOp} 
N_{\Omega} [\omega] := N_{\Omega} ^1 [\omega] + N_{\Omega} ^2 [\omega] + N_{\Omega} ^3 [\omega], 
\ee
where
\begin{align} \nonumber
N_{\Omega} ^1 [\omega] (t) &= \Psi_1 \Psi _2 \cdot K_{\Omega} [\omega (t)], \\  \nonumber
N_{\Omega} ^2 [\omega] (t) &= e^{-\frac{1}{\Wi} t}  \left ( \omega (0) - \Psi_1 \Psi_2 \cdot K_{\Omega} [\omega(0) ] \right ), \\ \nonumber
N_{\Omega} ^3 [\omega] (t) &= -\frac{\alp \Re}{ \St} \int_0 ^t e^{-\frac{1}{\Wi} (t-s)}  \Psi_2 \cdot K_{\Omega} [\omega (s) ] ds,
\end{align}
where and  $\Psi_1$ and $\Psi_2$ are smooth extensions of $\beta$ and $\hat{\tau}$, respectively satisfying that the boundary traces $T_{\partial \Omega} \Psi_1 = \beta ,$ and $T_{\partial \Omega} \Psi_2 = \hat{\tau}$ together with the support condition (with $\rho$ to be specified later in the proof)
\begin{align}
\mathrm{supp} (\Psi_i)  \subset E_{\rho} ( \partial \Omega) := \{ x \in \Omega \  |\  {\rm dist} (x, \partial \Omega) \le \rho \}, \quad i = 1, 2,
\end{align}
together with the estimate
\be\nonumber
\| D^k \Psi_i \|_{L^\infty (\Omega) } \le \frac{C}{\rho^k},\quad  i=1, 2, \quad k=0, 1, 2.
\ee
Note that
\begin{align} \label{Nestimate1}
\| N_{\Omega} ^1 [\omega] \|_{C_t H^2 (\Omega) \cap H_t ^1 H^1 (\Omega)} + \| N_{\Omega} ^3 [ \omega] \|_{C_t H^2 (\Omega) \cap H_t ^1 H^1 (\Omega) } \le C \left ( 1 + \frac{1}{\rho^2} \right ) \| \omega \|_{\mathcal{X} ' }, \\ \label{Nestimate2}
\| \Delta N_{\Omega} ^2 [\omega] \|_{C_t L^2 (\Omega) } + \| N_{\Omega} ^2 [ \omega] \|_{C_t H^1 (\Omega) \cap H_t ^1 H^1 (\Omega) } \le C \left ( 1 + \frac{1}{\rho^2 } \right ) \left ( \| \omega (0) \|_{H^1 (\Omega) } + \| \Delta \omega(0) \|_{L^2} \right ). 
\end{align}
Furthermore, by the Sobolev embedding $\| \omega(t) \|_{C(\bar{\Omega} )} \le C \| \omega(t) \|_{H^2 (\Omega) }$ and $\omega(0) \in C(\bar{\Omega})$, we have
\begin{align}
\| N_{\Omega} [\omega] \|_{C_t C(\bar{\Omega}) } \le C \left (1 + \frac{1}{\rho^2} \right ) \left ( \| \omega \|_{\mathcal{X}' } + \| \omega(0) \|_{H^1 (\Omega) \cap C(\bar{\Omega} ) } \right ).
\end{align}

\noindent{\textbf{Step 1: (Solution Scheme)} }
Let $F$ be an operator on $\mathcal{X}$ defined by $F(\omega) = \omega^n$, where $\omega^n$ is the solution of
\begin{align}\label{omsys1}
\partial_t \omega^n &= \frac{1}{\Re} \Delta \omega^n - K_{\Omega} [\omega] \cdot \nabla \omega^n + \nabla^\perp \cdot f_b, \quad\qquad \ \  \text{in} \quad   \Omega \times (0,T),\\
\omega^n (0) &= \omega(0),\qquad\qquad\quad\quad \qquad\qquad\qquad\quad\qquad\ \ \  \text{on} \quad   \Omega \times \{t=0\},\\
  T_{\partial \Omega}  [\omega^r]&= T_{\partial \Omega} [N_{\Omega} [\omega]]\quad\qquad\qquad\qquad\qquad\qquad\qquad \ \ \ \  \text{on} \quad   \partial\Omega \times (0,T).\label{omsys2}
\end{align}
Let $\omega^r = \omega^n - N_{\Omega} [\omega]$. Then $\omega^r$ solves
\begin{align} \label{omegar1}
\partial_t \omega^r &=  \frac{1}{\Re}  \Delta \omega^r - K_{\Omega} [\omega] \cdot \nabla \omega^r + R,\qquad\qquad\qquad\ \ \text{in} \quad   \Omega \times (0,T), \\
R &= \nabla^\perp \cdot f_b - \left(\partial_t + K_{\Omega} [\omega] \cdot \nabla -  \frac{1}{\Re}  \Delta \right) N_{\Omega} [\omega], \\
\omega^r (0) &= 0, \qquad\qquad\qquad\qquad\quad\quad \qquad\qquad\qquad\quad\ \ \  \text{on} \quad   \Omega \times \{t=0\}, \\ \label{omegar4}
  T_{\partial \Omega}  [\omega^r] &= 0 \qquad\qquad\qquad\quad\qquad\qquad\qquad\qquad\qquad \ \ \ \ \  \text{on} \quad   \partial\Omega \times (0,T).
\end{align}
Since $R \in L^2_t L^2 (\Omega)$ from previous calculations, there is a unique $\omega^r$ solving them, satisfying
\begin{align}
\omega^r \in C_t H_0 ^1 (\Omega) \cap L^2 _t (H^2 \cap H_0 ^1) (\Omega)  \cap H^1_t L^2 (\Omega).
\end{align}
As a consequence, we have
 \be
\omega^n = \omega^r + N_{\Omega} [\omega] \in C_t H^1 (\Omega) \cap H_t ^1 L^2 (\Omega),
\ee 
with $\omega^n (0) = \omega(0)$ and solves the system \eqref{omsys1}--\eqref{omsys2}. In addition, since $N_{\Omega} [\omega] \in C_t C (\bar{\Omega})$ by the maximum principle $\omega^n \in C_t C(\bar{\Omega})$. Note that we only used $\omega \in \mathcal{X}'$ and $\omega(0) \in C(\bar{\Omega})$ to obtain $F(\omega) = \omega^n \in \mathcal{X}$, and we do not need $\omega \in \mathcal{X}$. Finally, we note that $\Delta \omega^n = \Delta \omega^r + \Delta N_{\Omega} [\omega] \in L_t ^2 L^2 (\Omega)$. 
\vspace{2mm}

\noindent{\textbf{Step 2: (Contraction Mapping)}}
Next, we show that for a given $\omega_0 \in H^1 (\Omega) \cap C (\bar{\Omega})$ with $\Delta \omega_0 \in L^2 (\Omega)$, $F$ is in fact a contraction mapping in 
\begin{align}
\mathcal{Y} = \{ \omega \in \mathcal{X}' \ | \  \| \omega  \|_{\mathcal{X}' } \le B, \,\, \omega(0) = \omega_0 \},
\end{align}
for a suitable $B>0$, and small enough time $T$. Since we have enough regularity, we can rigorously perform the following calculation: for $\omega \in \mathcal{Y}$, let $v = F(\omega)$. Then
\begin{align} \label{veqn}
\partial_t v &=  \frac{1}{\Re}  \Delta v - K_{\Omega} [\omega] \cdot \nabla v + \nabla ^\perp \cdot f_b , \quad\quad\quad\quad   \text{in} \quad   \Omega \times (0,T),\\
v(0) &= \omega_0,\qquad\qquad\quad\quad \qquad\qquad\qquad\quad\qquad\ \ \  \text{on} \quad   \Omega \times \{t=0\}, \\ 
T_{\partial \Omega} v &= T_{\partial \Omega}  [ N_{\Omega} [\omega] ] \qquad\qquad\qquad\quad\qquad\qquad \ \ \ \ \ \  \text{on} \quad   \partial\Omega \times (0,T).
\end{align}
Since $\Delta v \in L^2_t L^2 (\Omega)$ we have
\begin{align} \label{LWellestim1}
\int_{\Omega} (- \Delta v ) \partial_t v dx +  \frac{1}{\Re}  \int_{\Omega} |\Delta v|^2 dx =  \int_{\Omega} K_{\Omega} [\omega] \cdot \nabla v (\Delta v) dx - \int_{\Omega} (\Delta v) (\nabla^\perp \cdot f_b) dx.
\end{align}
The first term of (\ref{LWellestim1}) becomes
\begin{align}\nonumber
-\int_{\Omega} \nabla \cdot (\nabla v \partial_t v ) dx +\frac{1}{2} \frac{\rmd}{\rmd t} \int_{\Omega}& |\nabla v | ^2 dx = - \int_{\partial \Omega} \hat{n} \cdot T_{\partial \Omega} (\nabla v ) T_{\partial \Omega} (\partial_t v) \rmd S +\frac{1}{2} \frac{\rmd}{\rmd t} \| \nabla v \|_{L^2 (\Omega)} ^2 \\  \nonumber
&= - \int_{\partial \Omega} \hat{n} \cdot T_{\partial \Omega}  (\nabla v ) T_{\partial \Omega} (\partial_t N_{\Omega} [\omega] ) \rmd S +\frac{1}{2} \frac{\rmd}{\rmd t} \| \nabla v \|_{L^2 (\Omega)} ^2 \\ \nonumber
&= - \int_{\Omega} \nabla \cdot ( \nabla v \partial_t N_{\Omega} [\omega] ) dx +\frac{1}{2} \frac{\rmd}{\rmd t} \| \nabla v \|_{L^2 (\Omega) } ^2 \\ 
&= - \int_{\Omega} \Delta v \partial_t N_{\Omega} [\omega] dx - \int_{\Omega} \nabla v \cdot \nabla \partial_t N_{\Omega} [\omega] dx +\frac{1}{2} \frac{\rmd}{\rmd t}  \| \nabla v \|_{L^2 (\Omega) } ^2 . \nonumber
\end{align}
The issue is control of $\| \partial_t N_{\Omega} [\omega] \|_{L^2 _t L^2 (\Omega)}$ and $\| \nabla \partial_t N_{\Omega} [\omega] \|_{L^2 _t L^2 (\Omega)}$. Here we use two tricks. 
\begin{enumerate}
\item We have a freedom in choosing $\rho$, and for small enough, fixed $T$ we choose $\rho = T^\beta$ accordingly. \vspace{2mm}
\item When controlling the term $\int_{\Omega} \nabla v \cdot \nabla \partial_t N_{\Omega} [\omega] dx$, we use
\begin{align} \nonumber
\int_{\Omega} \nabla v \cdot \nabla \partial_t N_{\Omega} [\omega] dx \le \| \nabla v \|_{L^2 (\Omega) } \| \nabla \partial_t N_{\Omega} [\omega] \|_{ L^2 (\Omega)} \le t^{-\alpha} \| \nabla v \|_{L^2 (\Omega) } ^2 + t^\alpha  \| \nabla \partial_t N_{\Omega} [\omega] \|_{ L^2 (\Omega)} ^2,
\end{align}
which enables the control of $ \| \nabla \partial_t N_{\Omega} [\omega] \|_{ L^2 (\Omega)} ^2$ term for a short time. 
\end{enumerate}
For $\| \partial_t N_{\Omega} [\omega] \|_{L^2 _t L^2 (\Omega)}$, we have
\begin{align} \nonumber
\| \partial_t &N_{\Omega} [\omega] (t) \|_{L^2 (\Omega)} \le C \| \omega_0 \|_{L^2 (\Omega) } \\ \nonumber
&\ + C \| \Psi_2 \|_{L^{p'} (\Omega)} \left (\| K_{\Omega} [\partial_t \omega (t)] \|_{L^p(\Omega) } + \| K_{\Omega} [\omega (t) ]\|_{L^p (\Omega) } + \int_0^t \| K_{\Omega} [\omega(s) ] \|_{L^p (\Omega) } ds \right ) \\
&\le C \| \omega_0 \|_{H^1 (\Omega)} + C T^{\frac{\beta}{p'}} \left ( \| \partial_t \omega (t) \|_{L^2 (\Omega) } + \| \omega(t) \|_{L^2 (\Omega) } + \sqrt{t} \| \omega\|_{L^2 _t L^2 (\Omega)} \right ),
\end{align}
where $\frac{1}{p} + \frac{1}{p'} = \frac{1}{2}$, $2 < p < \infty $ and $p' > 2$, by Sobolev embedding $H^1 (\Omega)\subset L^p (\Omega) $ and the bound $\| \Psi_i \|_{L^{p'} (\Omega) } \le C T^{\frac{\beta}{p'}}$. Similarly for $\| \nabla \partial_t N_{\Omega} [\omega] \|_{L^2 _t L^2 (\Omega)}$,
\begin{align} \nonumber
\| \nabla \partial_t N_{\Omega} &[\omega] (t) \|_{L^2 (\Omega)} \le \| \nabla (\Psi_1 \Psi_2 ) \|_{L^{p'} (\Omega) } \| K_{\Omega} [\partial_t \omega (t) ] \|_{L^p (\Omega) } + \| \Psi_1 \Psi_2 \|_{L^\infty (\Omega) } \| \nabla K_{\Omega} [\partial_t \omega (t) \|_{L^2 (\Omega) } \\ \nonumber
&\quad +\frac{1}{\Wi} \left ( \| \omega_0 \|_{H^1 (\Omega) } + \| \nabla (\Psi_1 \Psi_2 ) \|_{L^{p'} (\Omega) } \| K_{\Omega} [\omega_0 ] \|_{L^p (\Omega) } + \| \Psi_1 \Psi_2 \|_{L^\infty (\Omega) } \| \nabla K_{\Omega} [ \omega_0 \|_{L^2 (\Omega) } \right )  \\ \nonumber
&\quad+ \frac{\alp\Re}{\St}   \int_0 ^t  \left ( \| \nabla \Psi_2 \|_{L^2 (\Omega)} \| K_{\Omega} [\omega (s) ] \|_{L^\infty (\Omega) } + \| \Psi_2 \|_{L^\infty (\Omega) } \| \nabla K_{\Omega} [\omega(s)] \|_{L^2 (\Omega) } \right ) ds    \\ \nonumber
&\quad +\frac{\alp\Re}{\St}  \left ( \| \nabla \Psi_2 \|_{L^2 (\Omega) } \| K_{\Omega} [\omega(t) ] \|_{L^\infty (\Omega) } + \| \Psi_2 \|_{L^\infty (\Omega) } \| \nabla K_{\Omega} [\omega (t) ] \|_{L^2 (\Omega)} \right )  \\ 
&\le C (1 + T^{\beta(\frac{1}{p'} - 1)}) \left ( \| \omega_0 \|_{H^1 (\Omega) } + \| \partial_t \omega (t) \|_{L^2 (\Omega) } +  (1+t) \| \omega \|_{C_t H^1 (\Omega) } \right ),
\end{align}
by Sobolev embedding $  H^2 (\Omega)\subset L^\infty (\Omega)$ and the bounds $\| \nabla \Psi_2 \|_{L^2 (\Omega) } \le \| \nabla \Psi_2 \|_{L^\infty (\Omega) } T^{\frac{\beta}{2} }$ together with  $a^{\frac{1}{p'} -1 } > a^{-\frac{1}{2} }$ for $p'>2$ and $0<a<1$. Therefore, we have
\begin{align} \nonumber
\frac{\rmd}{\rmd t}& \| \nabla v \|_{L^2 (\Omega) } ^2 + \frac{1}{\Re} \| \Delta v \|_{L^2 (\Omega) } ^2 \le C (\| \omega \|_{C_t H^1 (\Omega) } ^2 + t^{-\alpha} ) \| \nabla v \|_{L^2 (\Omega) } ^2  \\ \nonumber
&+ C \left ( \| \omega_0 \|_{H^1 (\Omega) } ^2 +   \| f \|_{L_t ^\infty H^1 (\Omega)} ^2 + \| \omega_0 \|_{H^1 (\Omega) } ^2 + T^{\frac{\beta}{p'} } (\|\partial_t \omega(t) \|_{L^2 (\Omega) } ^2 + \| \omega(t) \|_{L^2 (\Omega) } ^2 + t \| \omega \|_{L^2 _t L^2 (\Omega) } ^2 ) \right ) \\ 
&+ C t^\alpha ( 1 + T^{\beta(\frac{1}{p'} - 1)} )^2 \left ( \| \omega_0 \|_{H^1 (\Omega) } ^2 + \| \partial_t \omega (t) \|_{L^2 (\Omega) } ^2 + (1+ t^2 ) \| \omega \|_{C_t H^1 (\Omega ) } ^2 \right ).
\end{align}
Noting from \eqref{veqn} that $ \| \partial_t v \|_{L^2 (\Omega) } ^2  \le  \Re^{-2} \| \Delta v \|_{L^2} ^2 + \| \omega \|_{C_t H^1 (\Omega)} ^2 \| \nabla v \|_{L^2 (\Omega )} ^2 +  \| f \|_{L_t ^\infty H^1 (\Omega)}^2$, 
and by Gr\"{o}nwall's inequality we have
\begin{align} \nonumber
\| v \|_{\mathcal{X}'} ^2 &= \|  v \|_{C_t H^1 (\Omega) } ^2  +  \| \partial_t v \|_{L^2 _t L^2 (\Omega) } ^2 \\ \nonumber
&\le C \exp \left ( T \| \omega \|_{C_t H^1 (\Omega) } ^2 + T^{1- \alpha } \right )\\ \nonumber
&\hspace{30mm}  \times \left ( T \|\omega_0 \|_{H^1(\Omega) }^2 + T  \| f \|_{L_t ^\infty H^1 (\Omega)}^2 + T(1 + T^\alpha (1 + T^{\beta(\frac{1}{p'} - 1) } ) ^2 ) \| \omega_0 \|_{H^1 (\Omega) } ^2 \right ) \\ \nonumber
&\qquad + C \exp \left ( T \| \omega \|_{C_t H^1 (\Omega) } ^2 + T^{1- \alpha } \right ) \\  \nonumber
&\hspace{25mm} \times \left ( (T ^{\frac{\beta}{p'} } (1 + T^2) + T^\alpha (1 + T ^{\beta(\frac{1}{p'} -1 )} ) ^2 ) \| \omega \|_{H^1_t L^2 (\Omega ) } ^2 + T (1 + T^2) \| \omega \|_{C_t H^1 (\Omega) } ^2 \right ) \\
&\le C e^{B^2 T +T^{1-\alpha } } O(T^q) \left (  \| \omega_0 \|_{H^1 (\Omega) } ^2 + \| f \|_{L_t ^\infty H^1 (\Omega)} ^2 + B^2  \right ),
\end{align}
where we choose $\alpha + 2 \beta (1 - \frac{1}{p'}) > 0$. Then for any $B>0$, for sufficiently small $T$ we have $\| v\|_{\mathcal{X}'} \le B$. \newline
The same calculation shows that $F$ is a contraction mapping on $\mathcal{Y}$ for a sufficiently small $T$. Let $\omega_1, \omega_2 \in \mathcal{Y}$ with $y = \omega_1 - \omega_2$, and let $z = F(\omega_1) - F(\omega_2)$. Then $z$ solves
\begin{align}
\partial_t z &= \frac{1}{\Re} \Delta z - K_{\Omega} [\omega_1] \cdot \nabla z - K_{\Omega} [y] \cdot \nabla F(\omega_2), \quad\quad\quad\ \  \text{in} \quad   \Omega \times (0,T),\\
z(0) &= 0,\qquad\qquad\quad\quad \qquad\qquad\qquad\quad\qquad\quad\qquad\quad\ \ \  \text{on} \quad   \Omega \times \{t=0\}, \\
T_{\partial \Omega} [z]&= T_{\partial \Omega}  [N_{\Omega} [y] ] \qquad\qquad\qquad\quad\quad\quad\qquad\qquad\qquad \ \ \ \ \ \  \text{on} \quad   \partial\Omega \times (0,T).
\end{align}
Then, the same computations as above gives the following bound on $z$ in $\mathcal{X}'$:
\begin{align}
\| z \|_{\mathcal{X}'} ^2 \le C \exp \left ( 2 T B + T^{1- \alpha} \right ) O(T^q)   \| y \|_{\mathcal{X}'} ^2 ( 1+ B^2 ),
\end{align}
which follows from the estimate
\begin{align} \nonumber
\| K_{\Omega} [y] \cdot \nabla F(\omega_2) \|_{L^\infty _t L^2 (\Omega) } ^2 \le C \| K_{\Omega} [y] \|_{C_t H^2 (\Omega)} ^2 \| F(\omega_2 ) \|_{C_t H^1 (\Omega) } ^2 \le C \| y \|_{\mathcal{X}'} ^2 B^2.
\end{align}
Consequently, there is unique $\omega \in \mathcal{X}'$ such that $F(\omega) = \omega$, and since $F(\omega) \in \mathcal{X}$ we have $\omega \in \mathcal{X}$. Then by Lemma \ref{lem1} and Lemma \ref{lem2} we have a bound 
\begin{align} \nonumber
\| \omega \|_{\mathcal{X} } \le C(\omega_0, T),
\end{align}
which does not blow up for finite $T>0$ or $\| \omega \|$. Also, $\Delta \omega (t) \in L^2 (\Omega)$ for a.e. $t \in [0, T]$, which means that we can continue the solution. Finally, this proves global well-posedness of the system in $\mathcal{X}$.
\end{proof}

\begin{cor} If $\omega_0 \in H^2 (\Omega)$, then $\omega \in L^2 (0, T; H^2 (\Omega))$.
\end{cor}
\begin{proof}
Note that $N_{\Omega} [\omega] \in C_t H^2 (\Omega)$ if $\omega_0 \in H^2 (\Omega)$ by estimates (\ref{Nestimate1}) and definition of $N_{\Omega} ^2 [\omega]$. Note that $\omega = \omega^r + N_{\Omega} [\omega]$, where $\omega^r$ solves the system (\ref{omegar1}) -(\ref{omegar4}), and therefore $\omega^r \in L^2 (0, T; H^2 (\Omega) )$.
\end{proof}

\section{Inviscid limit and quantitative drag reduction }\label{sec:dragred}

Consider a smooth solution $u$ of the Euler equations
\begin{align}\label{Eb}
\partial_t u + u \cdot \nabla u  &= - \nabla p   + f_b    \quad \qquad\quad \ \ \  \text{in} \quad   \Omega \times (0,T),\\
\nabla \cdot u  &= 0 \qquad \qquad\qquad\quad \ \ \ \ \  \text{in} \quad   \Omega \times (0,T),\\
u \cdot \hat{n} &= 0 \qquad \qquad\qquad\quad \ \ \ \ \  \text{on} \quad  \partial \Omega \times (0,T),\\ 
u|_{t=0} &= u_0   \qquad \qquad\qquad\quad \ \ \  \text{on} \quad  \Omega\times \{t=0\}.  \label{Ef}
\end{align}
Strong Euler solutions are guaranteed to exist globally starting from regular initial data in two-dimensions on domains with smooth boundaries \cite{Kato67}.  The nature of the inviscid  (high-Reynolds number) limit of solutions of the Navier-Stokes--End-Functionalized  polymer system \eqref{NSb}--\eqref{NSf} is a natural question; do solutions with infinitesimal viscosity behave approximately as strong solutions of the inviscid equations?  We answer this question in the affirmative below, and provide a rate of convergence as Reynolds number tends to infinity.

\begin{thm}[Inviscid Limit and Drag Reduction]\label{IL}
Let $\Omega\subset \mathbb{R}^2$  be a bounded domain with $C^2$ boundary. Fix $T>0$ and let  $u^\nu$ be a strong solution of \eqref{NSb}--\eqref{NSf} with initial data $u_0$ on $[0,T]\times \Omega$ and mean-zero forcing provided by Theorem \ref{thmglobal}.  Let  $u$ be the global strong Euler solution \eqref{Eb}--\eqref{Ef} with initial data $u_0$. Then 
\be
\sup_{t\in [0,T]} \|u^\nu(t)- u(t)\|_{L^2(\Omega)}  =O( \Re^{-1/2}).
\ee
Furthermore, the wall friction factor $\langle f\rangle$ (global momentum defect) vanishes as
\be
 \langle f\rangle :=\frac{1}{\Re} \fint_0^T \fint_{\partial\Omega} \hat{n} \cdot \nabla u^\nu(x,t) \rmd S\rmd t =   O(\Re^{-1}),
\ee
and the global energy dissipation tends to zero as
\be
\langle \varepsilon^\nu \rangle:= \frac{1}{\Re}\fint_0^T \fint_\Omega |\nabla u^\nu(x,t)|^2 \rmd x \rmd t  = O(\Re^{-1}).
\ee
\end{thm}

\begin{rem}[Scaling Limits]\label{restinvlim}

The Navier-Stokes -- End-Functionalized polymer system has four non-dimensional parameters, $\Re$, $\Wi$, $\alp$ and $\St$.   Our argument below shows that the key dimensionless quantities for passage to Euler in the inviscid limit and obtaining drag reduction are the following two ratios
\be\label{crit}
\alp= \frac{L}{R}, \qquad \frac{\alp \Re\Wi}{\St} =\alp \frac{\mu_p}{ \mu_s}.
\ee
Recall that
$\mu_s=\rho \nu$ is the dynamic solvent viscosity,  $\mu_p=N_P  \lambda k_B \ol{T}$ is the polymer viscosity, $\lambda=\zeta R^2 / 4 H k_B \ol{T}$ is the polymer relaxation time and $\zeta = 6 \pi \rho \nu a$ is the bead friction coefficient.  If the quantities \eqref{crit} behave well, say they are $O(\Re^\gamma)$ for some $\gamma<1$, then an inspection of our proof shows that the high-Reynolds number limit holds as $\Re\to \infty$, albeit with a slower rate of $\Re^{(\gamma-1)/2}$.

High Reynolds numbers can be achieved in practice either by taking viscosity small, taking the characteristic velocity $V$ large, taking large characteristic scales $L$, or some combination thereof.   Thinking of applications such as pipe of channel flow, one might think of $L$ as fixed\footnote{The pipes may be long in extend, but turbulent scales are set by the cross-sectional width which is not necessarily large.} and vary Reynolds number be either reducing the viscosity of the solvent of driving the fluid faster through the pipe by increasing the pressure head.

Let us analyze a few situations of varying Reynolds number $\Re$, paying attention to the ratio \eqref{crit}.
\begin{enumerate}
 \item  Perhaps the most practical of the potential limits is to hold $\nu$ and $L$ fixed and vary $V$.  In this case,
\begin{align}
\alp, \frac{\alp \Re\Wi}{\St} = O_V(1) = O({\Re}^{0}) \qquad \text{with} \qquad  \nu, L \ \ \text{fixed},
\end{align}
since neither $\alp$ nor ${\mu_p}/{ \mu_s}$ depend at all on the characteristic velocity $V$.
 \item If $L$ and $U$ are held fixed and $\nu$ is varied, recalling the Stokes--Einstein relation $\zeta = 6 \pi \rho \nu a$ we find ${\mu_p}/{ \mu_s}$ is independent of viscosity $\nu$. Consequently, 
\begin{align}
\alp,\frac{\alp \Re\Wi}{\St} = O_\nu(1) = O({\Re}^{0}) \qquad \text{with} \qquad  V, L \ \ \text{fixed}.
\end{align}
\item We cannot fix $V$ and $\nu$, and take $L$ large to increase Reynolds number.  This would result in $\alp=O(\Re)$ while the ratio ${\mu_p/}{ \mu_s}$ remains fixed, which is critical for our argument. 
\end{enumerate}
However, as remarked in Footnote 4, these limits should physically be interpreted as intermediate asymptotics.  In particular, decreasing viscosity will decrease the viscous sublayer of the flow near the wall, which is order $O(\nu)$.  Our tacit assumption is that the typical polymer length should be smaller than the gradient length of the flow which, near the wall, should be on the order of the sublayer.  Therefore, varying $\nu$ and keeping $R$ fixed is liable to break down when $R$ and the sublayer become of comparable sizes. 

 In order to maintain our effective continuum model description, one might consider performing a sequence of experiments where $R$ is decreased together with $\nu$ as $R=O({\Re}^{-\gamma})$ for $\gamma\in [0,1]$, while maintaining a sufficiently dense coating.  This requires, in particular, that the number density be taken of the order $N_P \sim R^{-(d-1)}$ where $d$ is the spatial dimension so that the continuous carpet approximation and mushroom regime remain valid.  For consistency,  since polymer length-scale itself is shrinking, the effective bead scale $a$ should be taken of order $O(R^\beta)$ for some $\beta\geq 1$. In that case, $\alp=L/R= O(\Re^{\gamma})$ and if $R$ is taken $O(\nu)$, then the ratio \eqref{crit} is order
\begin{align}
\alp = \Re^{\gamma},\quad \frac{\alp \Re\Wi}{\St} =  O({\Re}^{(d-2-\beta)\gamma}) \qquad \text{with} \qquad   V, L \ \ \text{fixed}.
\end{align}
Thus, provided that $\beta>0$ and $\gamma<1$, we again obtain inviscid limit while maintaining our continuum description for all viscosity.  The borderline case $\gamma=1$ is exactly parallel to the critical Navier-slip boundary conditions, see discussion in \cite{DN18wl}.  

In summary, taking the limit $\Re\to \infty$ either by modifying the viscosities of the fluids or their characteristic speeds, our Theorem \ref{IL} says that $u^\nu\to u$ the strong Euler solution and the wall-drag/ dissipation vanishes, at least in the regime of applicability of our macroscopic model.
\end{rem}

\begin{rem}
The conclusions of Theorem \ref{IL} extend in a straightforward manner for dimensions $d\geq 3$ on any time interval over which strong solutions $u^\nu$ of the Navier-Stokes--end-functionalized polymer system and strong Euler solutions $u$ exist.  Moreover, the initial conditions and forces need not be taken identical, strong convergence in $L^2$ suffices to pass to Euler in the inviscid limit.
\end{rem}

\begin{proof}
\noindent  \textbf{Step 1: Convergence to Euler.} 
Let $w=u^\nu-u$ be the difference of the two solutions.  Then
\begin{align*}
\partial_t w + w\cdot \nabla u + u^\nu \cdot \nabla w  &=-\nabla q+  \frac{1}{\Re} \Delta  u^\nu  \quad \qquad\quad  \text{in} \quad   \Omega \times (0,T),\\
\nabla \cdot w  &= 0 \quad\quad\quad\quad \quad\quad\quad\quad\quad\quad\text{in} \quad   \Omega \times (0,T),\\
w \cdot \hat{n} &= 0 \quad\quad\quad\quad \quad\quad\quad\quad\quad\quad\text{on} \quad  \partial \Omega \times (0,T),\\
w|_{t=0} &= 0   \quad \qquad \quad \qquad\qquad\quad\ \ \ \  \text{on} \quad   \Omega \times \{t=0\}.
\end{align*}
The energy in the difference field satisfies
\begin{align}
\partial_t\left(\frac{1}{2} |w|^2\right) + w\cdot \nabla u\cdot w+   \nabla\cdot  \left(\frac{1}{2} |w|^2u^\nu + q w\right) &=  \frac{1}{\Re} w\cdot\Delta  u^\nu.
\end{align}
Integrating and using the boundary conditions $u^\nu \cdot \hat{n} $ and $w \cdot \hat{n}$, we find
\begin{align}
\frac{1}{2}\frac{\rmd}{\rmd t} \|w\|_{L^2(\Omega)}^2 &\leq   \|\nabla u\|_{L^\infty(\Omega)} \|w\|_{L^2(\Omega)}^2 +     \frac{1}{\Re} \int_\Omega w\cdot\Delta  u^\nu dx.
\end{align}
Now first note that
\begin{align} \nonumber
 \int_\Omega w \cdot \Delta  u^\nu dx 
     &=   -   \|\nabla u^\nu\|_{L^2(\Omega)}^2   +  \int_\Omega \nabla u :\nabla   u^\nu dx +  \int_{\partial \Omega} w\cdot (\hat{n}\cdot \nabla) u^\nu \rmd S\\
     &\leq  - \frac{1}{2}  \|\nabla u^\nu\|_{L^2(\Omega)}^2 + \frac{1}{2 }  \|\nabla u\|_{L^2(\Omega)}^2 +   \int_{\partial \Omega} w\cdot (\hat{n}\cdot \nabla) u^\nu \rmd S.
\end{align}
Now note that for any tangential vector field to the boundary $v$ satisfying $v\cdot \hat{n}=0$ we have
\begin{align} \nonumber
\int_{\partial \Omega} v\cdot (\hat{n}\cdot \nabla) u^\nu \rmd S&= \int_{\partial \Omega} (v\cdot \hat{\tau}) ((\hat{n}\cdot \nabla) u^\nu\cdot \hat{\tau}) \rmd S 
\\ \nonumber
&=  \int_{\partial \Omega} (v\cdot \hat{\tau}) (2(D(u^\nu)n)\cdot \hat{\tau}) \rmd S-   \int_{\partial \Omega} (v\cdot \hat{\tau})  (\hat{\tau}\cdot \nabla n) \cdot u^\nu \rmd S\\ 
&= \int_{\partial \Omega} (v\cdot \hat{\tau}) (2(D(u^\nu)n)\cdot \hat{\tau}) \rmd S-  \int_{\partial \Omega} 2\kappa(v\cdot \hat{\tau})  ( u^\nu \cdot \hat{\tau}) \rmd S
\end{align}
where $\kappa=\hat \tau\cdot \nabla \hat n \cdot \hat \tau$ is the boundary curvature.
Combining with the boundary condition on Navier-Stokes
\be
u^\nu \cdot \hat\tau =- \frac{\St}{\alp\Re} \left(\partial_t  +\frac{1}{\Wi}\right) \left (2(D(u^\nu)  \hat{n})\cdot \hat\tau +\frac{\alp}{2} {u^\nu} \cdot \hat{\tau} \right ),
\ee
we have the following equality
\begin{align}\nonumber
\int_{\partial \Omega} u^\nu\cdot (\hat{n}\cdot \nabla) u^\nu \rmd S&= \int_{\partial \Omega} (u^\nu\cdot \hat{\tau}) \left(2(D(u^\nu)n)\cdot \hat{\tau}+ \frac{\alp}{2} {u^\nu} \cdot \hat{\tau} \right) \rmd S-  \int_{\partial \Omega} \left( \frac{\alp}{2}+2\kappa\right)(u^\nu\cdot \hat{\tau})^2 \rmd S\\ \nonumber
&=-\frac{\St}{\alp\Re} \frac{\rmd}{\rmd t}  \int_{\partial \Omega} |2(D(u^\nu)n)\cdot \hat{\tau}+ \frac{\alp}{2} {u^\nu} \cdot \hat{\tau}|^2 \rmd S\\
&\qquad   -\frac{\St}{\alp\Re\Wi}  \int_{\partial \Omega} |2(D(u^\nu)n)\cdot \hat{\tau}+ \frac{\alp}{2} {u^\nu} \cdot \hat{\tau}|^2 \rmd S-  \int_{\partial \Omega} \left( \frac{\alp}{2}+2\kappa\right)(u^\nu\cdot \hat{\tau})^2 \rmd S.
\end{align}
Consequently
\begin{align}\nonumber
\int_{\partial \Omega} w\cdot (\hat{n}\cdot \nabla) u^\nu \rmd S&=-\frac{\St}{\alp\Re} \frac{\rmd}{\rmd t}  \int_{\partial \Omega} |2(D(u^\nu)n)\cdot \hat{\tau}+ \frac{\alp}{2} {u^\nu} \cdot \hat{\tau}|^2 \rmd S \\ \nonumber
& -\frac{\St}{\alp\Re\Wi}  \int_{\partial \Omega} |2(D(u^\nu)n)\cdot \hat{\tau}+ \frac{\alp}{2} {u^\nu} \cdot \hat{\tau}|^2 \rmd S-  \int_{\partial \Omega} \left( \frac{\alp}{2}+2\kappa\right)(u^\nu\cdot \hat{\tau})^2 \rmd S \\ 
    & -  \int_{\partial \Omega} (u\cdot \hat{\tau}) \left(2(D(u^\nu)n)\cdot \hat{\tau}+ \frac{\alp}{2} {u^\nu} \cdot \hat{\tau}\right) \rmd S+   \int_{\partial \Omega} \left( \frac{\alp}{2}+2\kappa\right) (u\cdot \hat{\tau})  ( u^\nu \cdot \hat{\tau}) \rmd S.\label{bdryterm}
\end{align}
The Euler/Navier-Stokes cross-terms are handled  as follows.  First,
\begin{align} \nonumber
  \Bigg| \int_{\partial \Omega} &(u\cdot \hat{\tau}) \left(2(D(u^\nu)n)\cdot \hat{\tau}+ \frac{\alp}{2} {u^\nu} \cdot \hat{\tau}\right)  \rmd S\Bigg| \leq   \sqrt{ \int_{\partial \Omega} ( u \cdot \hat{\tau})^2 \rmd S \int_{\partial \Omega} |2(D(u^\nu)n)\cdot \hat{\tau}+ \frac{\alp}{2} {u^\nu} \cdot \hat{\tau}|^2 \rmd S}\\
&\qquad \leq  \frac{2\alp\Re\Wi}{\St}\int_{\partial \Omega} ( u \cdot \hat{\tau})^2 \rmd S + \frac{\St}{2\alp\Re\Wi}     \int_{\partial \Omega} |2(D(u^\nu)n)\cdot \hat{\tau}+ \frac{\alp}{2} {u^\nu} \cdot \hat{\tau}|^2 \rmd S. \label{crossbnd1}
\end{align}
The inequality \eqref{crossbnd1} allows us to hide the first cross-terms above.  As for the other cross-term, we note first that if $\alp > 4  \max_{x\in \partial \Omega}  \kappa$ (which is consistent with our assumption $(A_3)$), then this term is negative and can be dropped.  Otherwise, more generally we assume ${\alp}\neq 4\kappa$ and we have
\begin{align} \nonumber
 \left|\int_{\partial \Omega}\left( \frac{\alp}{2}+2\kappa\right) ( u^\nu \cdot \hat{\tau})( u \cdot \hat{\tau}) \rmd S \right|&\leq\frac{1}{2} \|{\alp}/{2}+2\kappa\|_{L^\infty(\partial\Omega)} \int_{\partial \Omega}  ( u^\nu \cdot \hat{\tau})^2 \rmd S\\
 &\qquad +   \frac{1}{2}\|{\alp}/{2}+2\kappa\|_{L^\infty(\partial\Omega)}\int_{\partial \Omega}  ( u \cdot \hat{\tau})^2 \rmd S .
\end{align}
We estimate the boundary term by trace inequality and embedding as follows
\begin{align*}
\int_{\partial \Omega}  ( u^\nu \cdot \hat{\tau})^2 \rmd S 
 \leq 4\|{\alp}/{2}+2\kappa\|_{L^\infty(\partial\Omega)} \|u^\nu\|_{L^2(\Omega)}^2  +\frac{ \|\nabla u^\nu\|_{L^2(\Omega)}^2}{4\|{\alp}/{2}+2\kappa\|_{L^\infty(\partial\Omega)} }.
\end{align*}
Thus, putting this together with \eqref{bdryterm} and \eqref{crossbnd1}  we find
\begin{align*}
\int_{\partial \Omega} w\cdot (\hat{n}\cdot \nabla) u^\nu \rmd S&\leq -\frac{\St}{\alp\Re} \frac{\rmd}{\rmd t}  \int_{\partial \Omega} |2(D(u^\nu)n)\cdot \hat{\tau}+ \frac{\alp}{2} {u^\nu} \cdot \hat{\tau}|^2 \rmd S \\ \nonumber
&\qquad -\frac{\St}{2\alp\Re\Wi}  \int_{\partial \Omega} |2(D(u^\nu)n)\cdot \hat{\tau}+ \frac{\alp}{2} {u^\nu} \cdot \hat{\tau}|^2 \rmd S\\ \nonumber
 &\qquad    +2\|{\alp}/{2}+2\kappa\|_{L^\infty(\partial\Omega)}^2 \|u^\nu\|_{L^2(\Omega)}^2+ \frac{1}{4} \|\nabla u^\nu\|_{L^2(\Omega)}^2
 \\
 &\qquad +  \left( \frac{1}{2}\|{\alp}/{2}+2\kappa\|_{L^\infty(\partial\Omega)} +\frac{2\alp\Re\Wi}{\St}\right)\int_{\partial \Omega} ( u \cdot \hat{\tau})^2 \rmd S.
\end{align*}
Finally, we obtain the following relative energy inequality
\begin{align} \nonumber
\frac{1}{2}\frac{\rmd}{\rmd t}&\left( \|w(t)\|_{L^2(\Omega)}^2 + \frac{\St}{\alp\Re^2}\int_{\partial \Omega} |2(D (u^\nu)\hat{n}) \cdot \hat\tau|^2 \rmd S \right) + \frac{1}{4\Re}  \|\nabla u^\nu\|_{L^2(\Omega)}^2   \\ \nonumber
 &\qquad \qquad + \frac{ \St}{2 \alp \Re ^2 \Wi } \int_{\partial \Omega} |2 (D(u^\nu ) \hat{n} ) \cdot \hat{\tau } |^2 dS \leq    \|\nabla u\|_{L^\infty(\Omega)} \|w(t)\|_{L^2(\Omega)}^2   + \frac{\mathcal{E}(t)}{\Re} ,\\
&\qquad\qquad\qquad\hspace{29mm}  \|w(0)\|_{L^2(\Omega)}^2= 0 \label{ineqVVlim}
\end{align}
where 
\begin{align} \nonumber
\mathcal{E}(t) &:=   \frac{1}{2 }  \|\nabla u\|_{L^2(\Omega)}^2 +2\|{\alp}/{2}+2\kappa\|_{L^\infty(\partial\Omega)}^2 \|u^\nu\|_{L^2(\Omega)}^2\\
&\qquad +   \left( \frac{1}{2}\|{\alp}/{2}+2\kappa\|_{L^\infty(\partial\Omega)} +\frac{2\alp\Re\Wi}{\St}\right)\int_{\partial \Omega} ( u \cdot \hat{\tau})^2 \rmd S.
\end{align}
Recalling Lemma \ref{lem0} for the bound on kinetic energy and working in the settings of (1) or (2) detailed in Remark \ref{restinvlim}, we have $\frac{\alp \Re\Wi}{\St} = O(\Re^0)$ and $\alpha = O(\Re^0)$ and thus
\be
\sup_{t\in [0,T]} \mathcal{E}(t)= O_{\Re}(1).
\ee
Integrating the above,  using Gr\"{o}nwall's inequality and the fact that $A>0$ we find for any $T>0$
\begin{align}\label{conv}
 \sup_{t\in [0,T]} \|u^\nu(t)- u(t)\|_{L^2(\Omega)}  =O( \Re^{-1/2}).
\end{align}
  Thus, we have convergence $ u^\nu \to u$ strongly in $L^\infty(0,T;L^2(\Omega))$.
  \vspace{2mm}

  \noindent  \textbf{Step 2: Vanishing of Wall Drag.} 
  The global momentum balance for Navier-Stokes reads
  \be \label{momentumbalance}
  \frac{\rmd}{\rmd t} \int_\Omega u^\nu \rmd x = -\int_{\partial\Omega} \hat{n} p^\nu \rmd S + \frac{1}{\Re} \int_{\partial\Omega} \partial_n u^\nu \rmd S.
  \ee
  The last term is the viscosity induced wall-friction, which we aim to show vanishes.  Indeed, using the divergence-free condition $\nabla\cdot u^\nu=0$ we have
  \be\label{boundarydivergence}
  \hat{n} \cdot \partial_n u^\nu|_{\partial\Omega} =  - \hat{\tau} \cdot \partial_\tau u^\nu|_{\partial\Omega}.
  \ee
  To see this, extend $\hat{n}(x)$ and $\hat{\tau}(x)$ smoothly into a tubular neighborhood of $\partial\Omega$ and such that they remain an orthonormal basis of $\mathbb{R}^2$.  Then expressing $\nabla = \hat{n} \partial_n + \hat{\tau}\partial_\tau$, forming $\nabla \cdot u =\hat{n} \partial_n u + \hat{\tau}\partial_\tau u$ and tracing on the boundary $\partial\Omega$ (recalling that $u\in L^\infty(0,T;H^2(\Omega))$, so that the trace makes sense), we obtain \eqref{boundarydivergence}.  Recalling also the identity \eqref{vortident} for vorticity along the walls
    \be
 \omega^\nu|_{\partial \Omega} = 2(D (u^\nu)\hat{n}) \cdot \hat\tau|_{\partial \Omega} + 2\kappa (u^\nu\cdot \hat{\tau})|_{\partial \Omega},
\ee
and returning to the wall-friction in \eqref{momentumbalance}, we have
  \begin{align} \nonumber
  \frac{1}{\Re} \int_{\partial\Omega} \partial_n u^\nu  \rmd S&=  \frac{1}{\Re} \int_{\partial\Omega} \hat{n} \cdot \partial_n u^\nu\ \hat{n}   \rmd S+  \frac{1}{\Re} \int_{\partial\Omega} \hat{\tau} \cdot \partial_n u^\nu \ \hat{\tau}\rmd S\\  \nonumber
  &=  -\frac{1}{\Re} \int_{\partial\Omega} \hat{\tau} \cdot \partial_\tau u^\nu\ \hat{n}   \rmd S+  \frac{1}{\Re} \int_{\partial\Omega} 2(D (u^\nu)\hat{n}) \cdot \hat\tau \ \hat{\tau}\rmd S-  \frac{1}{\Re} \int_{\partial\Omega}  \hat{n} \cdot \partial_\tau u^\nu \ \hat{\tau}\rmd S\\
  &= \frac{1}{\Re} \int_{\partial\Omega} (u^\nu \cdot \hat{\tau} ) \left[ \hat{\tau}\cdot \partial_\tau (\hat{\tau} \otimes  \hat{n}+\hat{n}\otimes \hat{\tau})  \right]\rmd S+ \frac{1}{\Re} \int_{\partial\Omega} 2(D (u^\nu)\hat{n}) \cdot \hat\tau \ \hat{\tau}\rmd S \label{fricexp}
  \end{align}
  Then, by trace theorem and the energy equality \eqref{energy}, we find for some $C:=C(\Omega,T, \frac{2 \alp\Re \Wi}{ \St})$ 
  \begin{align}\nonumber
      \frac{1}{\Re} \left|\int_0^T \int_{\partial\Omega} \partial_n u^\nu  \rmd S\rmd t\right| &\leq       \frac{C}{\Re}\| u^\nu\|_{L^\infty(0,T;L^2(\Omega))} \| \nabla u^\nu(t)\|_{L^2(0,T;L^2(\Omega))} + \frac{C}{\Re} \|(2D(u^\nu)\hat{n})\cdot \hat{\tau}\|_{L^2(0,T;L^2(\partial\Omega))}\\
      &\leq \frac{\Re \alp \Wi}{\St}\times \frac{C}{\Re} =O(\Re^{-1}),
  \end{align}
  where we used the bound \eqref{energybnd} and \eqref{ineqVVlim}.
  Note that the $L^\infty(0,T;L^2(\Omega))$ convergence established above implies that the pressure integrals must likewise converge.  Consequently,  the limiting global momentum balance
  reads: for any $0\leq t'\leq t\leq T$ 
  \be
  \int_\Omega u(t)\rmd x =   \int_\Omega u(t')\rmd x -\int_{t'}^t  \int_{\partial\Omega}  \hat{n} p(s)  \rmd S\rmd s.
  \ee
    \vspace{2mm}

\noindent  \textbf{Step 3: Vanishing of Energy Dissipation.} Finally note, directly from \eqref{conv} and \eqref{ineqVVlim} upon integration, 
\be\label{dissrate}
\frac{1}{\Re} \int_0^T \int_\Omega |\nabla u^\nu(x,t)|^2 dxd t\leq \frac{C(\frac{\alp \Re\Wi}{\St}, u_0, \Omega)}{\Re} .
\ee
This bound would hold also in higher dimensions, provided smooth Navier-Stokes-End-Functionalized polymer solution and Euler solutions exists on the a common time interval.
In two dimensions, the result follows again directly from the apriori bound on vorticity found in Lemma \ref{lem1}.  Specifically, using \eqref{BSbd} we have
\be
\frac{1}{\Re} \int_0^T \int_\Omega |\nabla u^\nu(x,t)|^2 dxd t \lesssim \frac{1}{\Re} \int_0^T \int_\Omega |\omega ^\nu(x,t)|^2 dxd t \leq  \frac{C(\frac{\alp \Re\Wi}{\St}, u_0, \Omega)}{\Re}.
\ee
  \end{proof}

\section{Discussion}

In this paper, we introduced a new boundary
condition for Navier-Stokes equations serves as a rational model for the situation where polymers are attached by one end (grafted or strongly adsorbed)
 to the wall in the ``mushroom regime" in which the polymers do not interact. This boundary condition was derived from a simple kinetic theory model for the polymer evolution on the boundary and
 a fluid-polymer stress balance.  Moreover, it closes in the macroscopic fluid variables and becomes an evolution equation for the tangential stresses on the solid walls. We established  global well-posedness for the resulting system in two spatial
dimensions and show that it captures the drag reduction effect in the sense that the vanishing viscosity
limit holds with a rate. Consequently, we obtain bounds on energy dissipation rate which qualitatively with observations of laminar drag reduction.

There are many further questions that are natural to ask.  These include, for example, the behavior in this system in higher dimensions, propagation of
higher regularity, and the resulting system for non-Hookean polymers (for example, for polymers modeled by the FENE potentials).   
 Another interesting direction of research, both scientifically and mathematically, concerns the validation, generalization and improvement of our assumptions. See Remark \ref{validity} for an extended discussion.  Perhaps the most interesting such generalization is to consider what occurs in  ``polymer brush"  regime in which the polymers are spaced close together on the boundary and may strongly interact with each other \cite{NA03,NA01}.  It is unclear to us whether or not a fully macroscopic description for this regime will be possible.  If not, a coupled microscopic-macroscopic system must be studied to understand the behavior in this regime.

\subsection*{Acknowledgments}  We are grateful to Peter Constantin, Gregory L. Eyink, Huy Q. Nguyen, Toan T. Nguyen, Chanwoo Kim, and Vlad Vicol for comments and discussions.
We would like to especially thank Ismail Hameduddin, Antonio Perazzo and  Tamer Zaki  for enlightening discussions on polymer physics. 
Research of TD is partially supported by NSF-DMS grant 1703997. Research of JL is partially supported by Samsung scholarship.
 
 \appendix
 
\section{Well-posedness theory of parabolic equations}

We recall some standard results on parabolic equations that we have used. Consider the problem 
\begin{align}  \label{heateqb}
\partial_t u + v \cdot \nabla u - \nu \Delta u& = f \quad \quad \text{in} \quad \Omega \times [0, T], \\ 
u &=0 \quad \quad \text{on} \quad \partial \Omega \times [0, T], \\
u|_{t=0} &= u_0 \quad \quad \!\! \text{on} \quad \Omega\times \{t=0\}, \label{heateq}
\end{align}
where $v \in C([0, T] ; C(\Omega) )$ with $\mathrm{div} \, v = 0$, and $\Omega$ is bounded with $C^2$ boundary. If $f \in L^2 (\Omega \times [0, T])$ and $u_0 \in H_0 ^1 (\Omega)$, then there is a unique solution of \eqref{heateqb}--(\ref{heateq}) satisfying
\begin{align} \nonumber
u &\in C ([0, T] ; H_0 ^1 (\Omega) ) \cap L^2 (0, T; H^2 (\Omega) \cap H_0 ^1 (\Omega ) ), \\ \nonumber
\partial_t u &\in L^2 (0, T; L^2 (\Omega) ).
\end{align}
For $v=0$ one can find this from Lions and Magenes \cite{LM12}  or Brezis \cite{Brezis}. For general $v$, one may follow the standard argument summarized below; for a full argument (see \cite{LM12} or \cite{C90}).
\begin{lemma}[Lions Projection Lemma] \label{LemmaLions}
Let $H$ be a Hilbert space and $\Phi \subset H$ a dense space. \\ Let $a: H \times \Phi \rightarrow \mathbb{R}$ be a bilinear form with the following two properties:
\begin{enumerate}
\item for all $\phi \in \Phi$, the linear form $u \rightarrow a(u, \phi)$ is continuous on $H$,
\item there is $\alpha > 0$ such that
\begin{align} \nonumber
a(\phi, \phi) \ge \alpha \| \phi \|_{H} ^2 \quad \quad \text{for all} \quad \phi \in \Phi.
\end{align}
\end{enumerate}
Then, for each continuous linear form $f \in H'$, there exists $u \in H$ such that 
\begin{align} \nonumber
\| u \|_{H} \le \frac{1}{\alpha} \| f\|_{H'}
\end{align}
and 
\begin{align} \nonumber
a(u, \phi) = \langle f, \phi \rangle \quad \quad \text{for all} \quad \phi \in \Phi.
\end{align}
\end{lemma}
\noindent To solve the system \eqref{heateqb}--(\ref{heateq}), we set 
\begin{align*}
H &= L^2 (0, T; H_0 ^1 (\Omega)),\\
\Phi &= \{ \phi = v|_{(0, T) \times \Omega } \ | \ v \in C_0 ^\infty ( (-\infty, T) \times \Omega ) \},\\
a(u, \phi) &= \int_{(0, T) \times \Omega } \left(\nabla u \cdot \nabla \phi - u \partial_t \phi - u v \cdot \nabla \phi \right) dxdt.
\end{align*}
Then,  Lemma \ref{LemmaLions} implies existence of solution of \eqref{heateqb}--(\ref{heateq}) in the weak sense and, together with
\begin{align} \nonumber
\int_{(0, T) \times \Omega}\left(\partial_t u v + u \partial_t v \right) dxdt  = \int_{\Omega \times \{t=T\} } uv dx - \int_{\Omega \times \{t=0\} } uv dx
\end{align}
and a standard density argument gives uniqueness. Finally, higher regularity follows from $v=0$ case with $f$ replaced by $f - v \cdot \nabla u \in L^2 (\Omega \times [0, T])$.

\section{Derivation of Kramers expression for polymer stresses}\label{AppendKramer}

Due to its central nature to our work,  we here provide a short derivation of Kramers expression (Eqn \eqref{Kramers}) for the polymer stresses for completeness.  The derivation is standard and can be found, for example in the textbook of Ottinger \cite{Ottinger12} on pages 158--159. 
We will calculate here only the components $ (-\hat{n}) \cdot \Sigma_P$, which are the force acting on the fluid in the direction normal to the wall.  This is the only component of the stress tensor used in our physical derivation and it has the most intuitive interpretation. 

First note that, within the bead-spring approximation, a polymer can exert force on a fluid parcel if and only if its end bead is contained in that fluid parcel.  Thus, we set up a cut-off between polymer layer and fluid parcel. In other words, we imagine a tubular neighborhood along the wall of size $\ell$. The thickness (in the wall-normal direction) of the near-wall fluid parcel acted upon by the polymer has characteristic size on the order of $r$, the maximal extent of the polymer defined in assumption $(A_4)$. Its length-scale in the wall-tangential direction is taken larger than that of the typical polymers.  As a consequence, the bead does not belong to the fluid parcel only if $ (-\hat{n}) \cdot m < \ell$. 
The thickness scale is justified since we are interested in the fluid parcel directly communicating with polymer. 
Let $(-\hat{n}) \cdot \Sigma_P^{\ell}$ be the (spring) force per surface, divided by solvent density. This is the force that polymers exert on the near-wall fluid parcel sitting at distance $\ell$ uniformly from the wall. Fixing $\ell$, this force is mathematically expressed as
\be
(-\hat{n}) \cdot \Sigma_P ^{\ell} = r \int_{M(x)} \chi_{\{ (-\hat{n}) \cdot m \ge \ell\}}(m)   \frac{k_B \overline{T}}{\rho} \nabla_m U(m) f_L(m) dm.
\ee
However, we note the following: there is no obvious choice for cut-off distance $\ell$ for polymer layer and fluid particles. Thus, to obtain the cumulative force $(-\hat{n}) \cdot \Sigma_L$, we average over possible scales $\ell$ and obtain
\be
(-\hat{n}) \cdot \Sigma_P = \frac{1}{r} \int_0 ^r (-\hat{n}) \cdot \Sigma_P ^{\ell} d \ell.
\ee
In the case of the Hookean dumbbell model for which $r = \infty$  which can be understood in suitable limiting sense. We do not detail this here. Therefore,
\begin{align} \nonumber
(-\hat{n}) \cdot \Sigma_P &= \frac{k_B \overline{T}}{\rho}  \int_0 ^r  \int_{M(x)} \chi_{\{ (-\hat{n}) \cdot m \ge \ell\}} \nabla_m U f_L(m) dm d \ell \\
&=  \frac{k_B \overline{T}}{\rho}   \int_{M(x)} \int_0 ^{ (-\hat{n}) \cdot m}  \nabla_m U f_L(m) d \ell dm = \frac{k_B \overline{T}}{\rho}  \int_{M(x)} (-\hat{n}) \cdot m \nabla_m U f_Pdm.
\end{align}
We thereby recover the Kramer formula \eqref{Kramers} for the normal component of polymer stress along the wall.


\end{document}